\documentclass{amsart}

\usepackage{etex}

\usepackage{hyperref}
\usepackage{txfonts, amsmath,amstext,amsthm,amscd,amsopn,verbatim,amssymb, amsfonts, amsrefs}
\usepackage{fullpage}
\usepackage{stmaryrd}

\usepackage{mathtools}
\usepackage[hang,small,center]{caption}
\usepackage[bbgreekl]{mathbbol}
\usepackage{graphicx}
\usepackage{xypic}

\usepackage{tikz}

\usetikzlibrary{matrix}
\usetikzlibrary{shapes}
\usetikzlibrary{arrows}
\usetikzlibrary{calc,3d}
\usetikzlibrary{decorations,decorations.pathmorphing}
\usetikzlibrary{through}
\tikzset{ext/.style={circle, draw,inner sep=1pt},int/.style={circle,draw,fill,inner sep=1pt},nil/.style={inner sep=1pt}}
\tikzset{exte/.style={circle, draw,inner sep=3pt},inte/.style={circle,draw,fill,inner sep=3pt}}
\tikzset{diagram/.style={matrix of math nodes, row sep=3em, column sep=2.5em, text height=1.5ex, text depth=0.25ex}}
\tikzset{diagram2/.style={matrix of math nodes, row sep=0.5em, column sep=0.5em, text height=1.5ex, text depth=0.25ex}}

\usepackage{tikz-cd}

\usepackage{bbm} 

\theoremstyle{plain}

\newtheorem{thm}{Theorem}[section]

\newtheorem{defn}[thm]{Definition}
\newtheorem{prop}[thm]{Proposition}

\newtheorem{lemma}[thm]{Lemma}

\theoremstyle{definition}

\newtheorem{ex}[thm]{Example}
\newtheorem{rem}[thm]{Remark}

\newcommand{\p}{\partial}

\newcommand{\R}{{\mathbb{R}}}

\newcommand{\op}{\mathcal}

\newcommand{\bpm}{\begin{pmatrix}}
\newcommand{\epm}{\end{pmatrix}}

\DeclareMathOperator{\Emb}{Emb}
\DeclareMathOperator{\Imm}{Imm}
\DeclareMathOperator{\Embbar}{\overline{Emb}}

\DeclareMathOperator{\Diff}{Diff}

\newcommand{\lD}{\mathcal{D}}

\newcommand{\Aut}{\mathrm{Aut}}

\DeclareMathOperator{\TopCat}{Top}

\newcommand{\beq}[1]{\begin{equation}\label{#1}}
\newcommand{\eeq}{\end{equation}}

\newcommand{\Operad}{\mathrm{Operad}}
\newcommand{\Bimod}{\mathrm{Bimod}}
\newcommand{\St}{\mathrm{V}}
\newcommand{\hofiber}{\mathrm{hofiber}}
\newcommand{\fr}{{\mathrm{fr}}}
\newcommand{\SO}{\mathrm{SO}}
\newcommand{\TOP}{\mathrm{TOP}}
\newcommand{\OO}{\mathrm{O}}

\begin{document}
\title{On the delooping of (framed) embedding spaces}

\author{Julien Ducoulombier}
\address{Max Planck Institute for Mathematics \\ Vivatsgasse 7 \\
R\"amistrasse 101 \\
53111 Bonn, Germany}
\email{julien@mpim-bonn.mpg.de}

\author{Victor Turchin}
\address{Department of Mathematics\\
Kansas State University\\
138 Cardwell Hall\\
Manhattan, KS 66506, USA}
\email{turchin@ksu.edu}

\author{Thomas Willwacher}
\address{Department of Mathematics \\ ETH Zurich \\
R\"amistrasse 101 \\
8092 Zurich, Switzerland}
\email{thomas.willwacher@math.ethz.ch}

\date{}

\thanks{
The authors acknowledge the University of Lille for hospitality. V.T.  has benefited from a visiting position of the Labex CEMPI (ANR-11-LABX-0007-01) at the Universit\'e de Lille and from a visiting position at the Max Planck Institute for Mathematics, Bonn for the achievement of this work. He was also partially supported by the Simons Foundation  grant ID:519474.
T.W. and J.D. have been partially supported by the NCCR SwissMAP funded by the Swiss National Science Foundation and the ERC starting grant GRAPHCPX}


\begin{abstract}
It is known  that the bimodule derived mapping spaces between two operads have a delooping in terms of the operadic mapping space. We show a relative version of that statement.
The result has applications to the spaces of disc embeddings fixed near the boundary and framed disc embeddings.
\end{abstract}

\maketitle

\vspace{-15pt}

\section{Introduction}

Let $\op P$ and $\op Q$ be  topological operads satisfying some mild conditions to be detailed below.
Suppose furthermore that we have a map of (pointed) spaces from some space $X$ to the operadic mapping space
\[
X\to \Operad(\op P, \op Q).
\]
Then one may in particular form the $\op P$-bimodule $\op Q{\circ} X$ defined such that $(\op Q{\circ} X)(n)=\op Q(n)\times X^{\times n}$, for which one uses the basepoint to define the left $\op P$-action and the map from $X$ to the mapping space to define the right $\op P$-action (see Section~\ref{s2}).
Our main result Theorem~\ref{th:main} is then that the following is a homotopy fiber sequence
\beq{equ:main}
\Bimod^h_{\op P}(\op P, \op Q{\circ} X) \to X \to \Operad^h(\op P, \op Q),
\eeq
where the superscript $h$ is used to show that we consider the derived version of the corresponding mapping spaces.
This result can be considered as a generalisation of the delooping result \cite{DwyerHess0,Duc} 
\beq{eq:ducoul}
\Bimod^h_{\op P}(\op P, \op Q) \simeq \Omega\Operad^h(\op P, \op Q),
\eeq
which can be recovered by setting $X=*$ to be a point. It is shown by the first author in~\cite{Duc} that~\eqref{eq:ducoul} is an equivalence of 
algebras over the little segments operad $\lD_1$. In Section \ref{Final2} we improve this result and  build an explicit weak equivalence of $\mathcal{SC}_{1}$-algebras between the pairs
$$
\left(
\begin{array}{c}\vspace{5pt}
\Omega\Operad^h(\op P, \op Q)) \\ 
hofiber(X\to  \Operad^h(\op P, \op Q))
\end{array} 
\right) \longrightarrow 
\left(
\begin{array}{c}\vspace{5pt}
\Bimod^h_{\op P}(\op P, \op Q) \\ 
\Bimod^h_{\op P}(\op P, \op Q{\circ} X)
\end{array} 
\right)
$$
where $\mathcal{SC}_{1}$ is the one dimensional Swiss-Cheese operad.

 We propose three applications of the above results. \vspace{10pt}

\noindent \textbf{Application 1: The space of disc embeddings.}
Consider the space $\Emb_\p(D^m, D^n)$ of disc embeddings  fixed in a neighbourhood of the boundary to be the standard equatorial inclusion $S^{m-1}\subset S^{n-1}$.
Assume furthermore $n-m\geq 3$ throughout. Let also $\lD_k$ denote the little $k$-discs operad.
Then the embedding space has two known deloopings, which we shall briefly describe.
First, one considers the homotopy fiber over immersions
\[
\Embbar_\p(D^m, D^n) =\hofiber\bigl( \Emb_\p(D^m, D^n) \to \Imm_\p(D^m, D^n)\simeq \Omega^m\St_{m,n}\bigr),
\] 
where $\St_{m,n}$ is the Stiefel manifold. It has been shown in \cite{WBdB2} that $\Embbar_\p(D^m, D^n)\cong \Omega^{m+1}\Operad^h(\lD_m,\lD_n)$, and that furthermore 
\beq{eq:BW1}
\Emb_\p(D^m,D^n) \simeq \Omega^m \hofiber(\St_{m,n}\to \Operad^h(\lD_m,\lD_n)).
\eeq
A second delooping is obtained in \cite{DucT}, where it is shown that 
\beq{eq:DucT1}
\Emb_\p(D^m,D^n) \simeq \Omega^m \Bimod^h_{\lD_m}(\lD_m,\lD_n^{m-\fr}), 
\eeq
where $\lD_n^{m-\fr}$ is the bimodule of $m$-framed little $n$-disks, which one should think of as embeddings of $m$-dimensional disks in the unit $n$-disk.

Our result \eqref{equ:main} above with $X=\St_{m,n}$ then shows that both deloopings agree:
\beq{eq:cor1}
\hofiber(\St_{m,n}\to \Operad^h(\lD_m,\lD_n)) \simeq \Bimod^h_{\lD_m}(\lD_m,\lD_n^{m-\fr}).
\eeq \vspace{-1pt}

\noindent \textbf{Application 2: The space of framed disc embeddings.}
Next consider the case that the operad $\op Q$ is acted upon by a topological group $G$.
Assuming that we have some map $f:\op P\to \op Q$ we hence obtain a map
\[
G \to \Operad(\op P,\op Q)
\] 
by composing $f$ and the $G$-action. 
The result \eqref{equ:main} in this case yields the first items of the fiber sequence
\[
\Bimod_{\op P}^h(\op P,\op Q{\circ} G) \to G \to \Operad^h(\op P, \op Q) \to \Operad^h(\op P, \op Q) \sslash G.
\]
Since in this case the fiber sequence may be extended as shown we obtain the delooping
\beq{equ:bimoddeloop}
\Bimod_{\op P}^h(\op P,\op Q{\circ} G) \simeq \Omega(\Operad^h(\op P, \op Q) \sslash G).
\eeq

Note that in this case the $\op P$-bimodule $\op Q{\circ} G$ is in fact an operad. However, the 
equivalence~\eqref{eq:ducoul} holds provided $\op Q(1)\simeq *$ and therefore it might  not be
true if the operad $\op Q$ is replaced by $\op Q{\circ} G$. In fact in case
$\op P(0)=\op Q(0)=*\simeq \op P(1)\simeq \op Q(1)$, which we will be assuming throughout the paper, and assuming that $G$ is connected and $\nsimeq *$, one has that the derived mapping spaces of bimodules 
\begin{equation}\label{eq:fr_nfr}
\Bimod_{\op P}^h(\op P,\op Q{\circ} G)\nsimeq\Bimod_{\op P}^h(\op P,\op Q)
\end{equation}
are not  weakly equivalent. By contrast, one has
\begin{equation}\label{eq:fr_nfr2}
\Operad^h(\op P, \op Q{\circ}G)\simeq \Operad^h(\op P,\op Q),
\end{equation}
see Remark~\ref{r:framed_bim_oper}

We apply the above findings to the spaces of framed disc embeddings $\Emb_\p^\fr(D^m, D^n)$.
It is shown in \cite{DucT}  that 
\beq{eq:DucT2}
\Emb_\p^\fr(D^m, D^n) \simeq \Omega^m \Bimod^h_{\lD_m}(\lD_m, \lD_n^{\fr}),
\eeq
where $\lD_n^{\fr}$ denotes the operad of positively framed little $n$-discs. Applying \eqref{equ:bimoddeloop} we hence obtain the $(m+1)$-st delooping
\beq{eq:cor2}
\Emb_\p^\fr(D^m, D^n) \simeq \Omega^{m+1}(\Operad^h(\lD_m, \lD_n) \sslash \SO(n)).
\eeq 

One should mention that the  delooping of the space of framed knots~\eqref{eq:cor2} that we discovered, can  be alternatively proved from the Boavida-Weiss result~\eqref{eq:BW1}, as explained in the last section (see Remark~\ref{r:sakai2}).

\bigskip

\noindent \textbf{Application 3: The Goodwillie-Weiss calculus and the smoothing theory.}
The deloopings~\eqref{eq:BW1}, \eqref{eq:DucT1}, \eqref{eq:DucT2} were obtained in~\cite{WBdB2,DucT} using the Goodwillie-Weiss
functor calculus on manifolds. In fact one obtains there the deloopings of the  Taylor towers
$T_k\Emb_\p(D^m,D^n)$ and $T_k\Emb_\p^\fr(D^m,D^n)$, $1\leq k\leq \infty$,  (without any codimension restriction on $m$ and $n$) by taking the derived mapping spaces of $k$-truncated operads and bimodules. Similarly~\eqref{eq:ducoul} and our main result~\eqref{equ:main} also have a truncated version:
\beq{equ:main2}
\Bimod^h_{\op P;\leq k}(\op P_{\leq k}, \op (Q{\circ} X)_{\leq k}) \to X \to \Operad^h_{\leq k}(\op P_{\leq k}, \op Q_{\leq k})
\eeq
is a homotopy fiber sequence for any $k\geq 1$, see Theorem~\ref{th:main}.

The obtained delooping result is of a particular interest when $m=n$:
\beq{eq:cor3}
T_\infty\Emb_\p(D^n, D^n) \simeq \Omega^{n+1}(\Aut^h(\lD_n) \sslash \SO(n)),
\eeq 
which should be compared to the Morlet-Burghelea-Lashof delooping of the group of 
relative to the boundary disc diffeomorphisms~\cite{BuL}:
\[
\Diff_\p(D^n)\simeq\Omega^{n+1}(\TOP(n)/\OO(n)),\,\, n\neq 4.
\]

In~\cite{Sakai} based on the Burghelea-Lashof work, K.~Sakai produces a similar delooping of the space
$\Emb_\p(D^m,D^n)$. In  the last Section~\ref{s:last} we show how this smoothing theory delooping 
can be adjusted to get a delooping of the space $\Emb^\fr_\p(D^m,D^n)$ of framed embeddings.
It is a very intriguing question whether the smoothing theory deloopings of spaces of long embeddings agree with the operadic ones
arising from the Goodwillie-Weiss calculus in the case when both deloopings are available (i.e., for $n-m\geq 3$ and $n\geq 5$).

\bigskip

For other related results on the little discs action on the spaces of disc embeddings and results on their deloopings,
we refer the reader to~\cite{BatDL,Budney1,Budney2,Duc2,Mostovoy,MoriyaSakai,Salvatore1,Turchin5}.


\section{The Reedy model categories of reduced operads and bimodules}\label{s2}

In this section, we cover the notion of a (truncated) operad and a (truncated)  bimodule over an operad. We equip these two categories with model category structures, called Reedy model category structures, using  left adjoints of the forgetful functors to the model category of $\Lambda$-sequences. For a more detailed account about the category of $\Lambda$-sequences and the Reedy model category of reduced operads, we refer the reader to \cite{Fr}. A precise study of the Reedy model category of reduced bimodules can be found in \cite{DucT,DucTF}.

\subsection{The model category of $\Lambda$-sequences} Let $\Lambda$ be the category whose objects are finite sets $[n]:=\{1,\ldots, n\}$, with $n\geq 1$, and morphisms are injective maps between them. By a $\Lambda$-sequence, we understand a functor $Y:\Lambda^{op}\rightarrow \TopCat$. By convention, we denote by $Y(n)$ the space $Y([n])$. In practice, a $\Lambda$-sequence $Y$ is a family of spaces $Y(1),\, Y(2),\ldots$ together with operations of the form
$$
u^{\ast}:Y(n)\rightarrow Y(m),\hspace{20pt}\text{ for any } u\in \Lambda([m]\,;\,[n]).
$$
A $\Lambda$-sequence $Y$ is said to be pointed if the space $Y(1)$ is equipped with a basepoint. 

Following \cite{Fr}, the categories $\Lambda Seq$ and $\Lambda Seq^{\ast}$ of $\Lambda$-sequences and pointed $\Lambda$-sequences, respectively, are endowed with model category structures in which a natural transformation $f:Y\rightarrow Z$ is a weak equivalence if it is an objectwise weak homotopy equivalence. Furthermore, a natural transformation $f$ is a fibration if the maps $Y(n)\rightarrow \mathcal{M}(Y)(n)\times_{\mathcal{M}(Z)(n)}Z(n)$, with $n\geq 1$, are Serre fibrations. The space $\mathcal{M}(Y)(n)$, called matching object of $Y$, is given by the formula 
\begin{equation}\label{B3}
\mathcal{M}(Y)(n):=\underset{\substack{\Lambda_{+}([i]\,;\,[n])\\ i< n}}{lim}\,\, Y(i),
\end{equation}
where $\Lambda_{+}$ is the subcategory of $\Lambda$ consisting of order-preserving maps. Similarly, for $k\geq 1$, the category of $k$-truncated $\Lambda$-sequences $\Lambda Seq_{\leq k}$ (resp. $k$-truncated  pointed $\Lambda$-sequences $\Lambda Seq^{\ast}_{\leq k}$), whose objects are functors $Y\in \Lambda Seq$ (resp.  $Y\in \Lambda Seq^{\ast}$) having $Y(n)=\emptyset$ for all $n>k$, inherits a model category structure. 

\medskip

\begin{defn}\label{d:circle}
Given a topological space $X$, define the $\Lambda$-sequence $X^{\times\bullet}$ assigning to $[n]$
the space $X^{\times n}$ of maps $[n]\to X$. The $\Lambda$-action is defined by precomposition: for any $u\in \Lambda([m]\,,\,[n])$, 
\[
u^\ast\colon (x_1,\ldots,x_n)\longmapsto (x_{u(1)},\ldots,x_{u(m)}).
\]
For a $\Lambda$-sequence $Y$,  define a $\Lambda$-sequence 
$Y{\circ}X$ as an objectwise product of $Y$ and $X^{\times\bullet}$.
\end{defn}

\begin{lemma}\label{l:circle}
In case $Y$ is a Reedy fibrant $\Lambda$-sequence, $X$ is any space, the $\Lambda$-sequence 
$Y{\circ}X$ is also Reedy fibrant.
\end{lemma}

\begin{proof}
It is easy to see that $X^{\times\bullet}$ is a Reedy fibrant $\Lambda$-sequence. Indeed, we have that
\[
\mathcal{M}(X^{\times\bullet})(n)=
\begin{cases}
*,&n=1\\
X^{\times n},&\text{otherwise}
\end{cases}.
\]
Thus $X^{\times n}\to  \mathcal{M}(X^{\times\bullet})(n)$ is always a Serre fibration. On
the other hand, the objectwise product of two Reedy fibrant $\Lambda$-sequences is so as well.
\end{proof}

\begin{defn}
By a $\Sigma${\rm -cofibrant} object, we understand a $\Lambda$-sequence $X$ such that each space $X(n)$, with $n\geq 1$, is cofibrant in the model category $\TopCat_{\Sigma_{k}}$ of spaces equipped with an action of the symmetric group $\Sigma_{k}$. Fibrations and weak equivalences for this model structure are objectwise Serre fibrations and objectwise weak equivalences. 
\end{defn}

\subsection{The Reedy model category of reduced operads} A reduced operad $\op O$ is a pointed $\Lambda$-sequence $\op O: \Lambda^{op}\rightarrow \TopCat$ together with operations, called \textit{operadic compositions}, of the form
\begin{equation}\label{B4}
{\circ}_{i}:\op O(n)\times \op O(m)\longrightarrow \op O(n+m-1),\hspace{15pt} \text{with } 1\leq i\leq n,
\end{equation}
satisfying compatibility with the $\Lambda$-structure on $\op O$, associativity and unit axioms \cite[Part II Section 8.2]{Fr}. A map between reduced operads should respect the operadic compositions. We denote by $\Operad$ the category of reduced operads. Note that $\Operad$ is equivalent to the full subcategory of topological operads having a point as
an arity-zero component. In what follows, we use the notation
$$
x{\circ}_{i}y={\circ}_{i}(x,y),\hspace{20pt}\text{for all } x\in O(n) \text{ and } y\in O(m).
$$

Given an integer $k\geq 1$, we also consider the category of $k$-truncated reduced operads denoted by $\Operad_{\leq k}$. The objects are $k$-truncated pointed $\Lambda$-sequences together with operations of the form (\ref{B4}) with $n+m\leq k+1$. Furthermore, one has the following functor called the $k$-truncation functor:
$$
\begin{array}{rcl}\vspace{3pt}
(-)_{\leq k}:\Operad & \longrightarrow & \Operad_{\leq k} \\ 
 \op O & \longmapsto & \op O_{\leq k}:=\left\{ 
 \begin{array}{ll}\vspace{3pt}
 \op O_{\leq k}(n)=\op O(n) & \text{if } n\leq k, \\ 
 \op O_{\leq k}(n)=\emptyset & \text{if } n > k.
 \end{array} 
 \right.
\end{array} 
$$

For $k\geq 1$, the categories $\Operad$ and $\Operad_{\leq k}$ are endowed with the so called Reedy model category structures transferred from $\Lambda Seq^{\ast}$ and $\Lambda Seq^{\ast}_{\leq k}$, respectively, along the adjunctions 
$$
\begin{array}{rcl}\vspace{3pt}
\mathcal{F}_{Op}: \Lambda Seq^{\ast} & \leftrightarrows & \Operad:\mathcal{U}, \\ 
\mathcal{F}_{Op\,;\,\leq k}: \Lambda Seq^{\ast}_{\leq k} & \leftrightarrows & \Operad_{\leq k}:\mathcal{U},
\end{array} 
$$ 
where $\mathcal{U}$ is the forgetful functor while $\mathcal{F}_{Op}$ and $\mathcal{F}_{Op\,;\,\leq k}$ are the free operadic functors. In other words, a map $f:\op P \rightarrow \op Q$ of (possibly truncated) operads is a weak equivalence (resp. a fibration) if the corresponding map $\mathcal{U}(f)$ is a weak equivalence (resp. a fibration) in the model category of (possibly truncated) pointed $\Lambda$-sequences.  

\begin{ex}\textit{The framed operad $\op O{\circ} G$}

\noindent Let $G$ be a topological group and $\op O$ be a reduced operad for which each space $\op O(n)$ admits an action of $G$ compatible with the $\Lambda$ structure and the operadic compositions. Then, the $\Lambda$-sequence $\op O{\circ} G$,
see Definition~\ref{d:circle}, inherits an operadic structure from the operad $\op O$ and the group structure of $G$. 
The operadic compositions are given by the following formula:
$$
\begin{array}{rlcl}
{\circ}_{i}:&\op O{\circ} G(n)\times \op O{\circ} G(m) & \longmapsto & \op O{\circ} G(n+m-1); \\ 
&(\theta;g_{1},\ldots,g_{n}) \,;\, (\theta';g'_{1},\ldots,g'_{n})& \longmapsto & (\theta {\circ}_{i} (g_{i}\cdot \theta'); g_{1},\ldots,g_{i-1},g_{i}g'_{1},\ldots,g_{i}g'_{m},g_{i+1},\ldots,g_{n}).
\end{array} 
$$
\end{ex}

\begin{ex}\textit{The little discs operads $\lD_{m}$}

\noindent In arity $n$, the space $\lD_{m}(n)$ is the configuration space of $n$ discs of dimension $m$, labelled by $[n]$, inside the unit disc of dimension $m$ having disjoint interiors. The unit in arity $1$ is given by the identity map. The $\Lambda$-structure is obtained by removing some discs and permuting the other ones. Finally, the operadic composition ${\circ}_{i}$ substitutes the $i$-th disc of the first configuration by the second configuration as illustrated in Figure \ref{B5}. In particular, each space $\lD_{m}(n)$ admits an action of $\SO(m)$ and we denote by $\lD_{m}^{\fr}$ the corresponding framed operad. \vspace{-10pt}
\begin{figure}[!h]
\begin{center}
\includegraphics[scale=0.2]{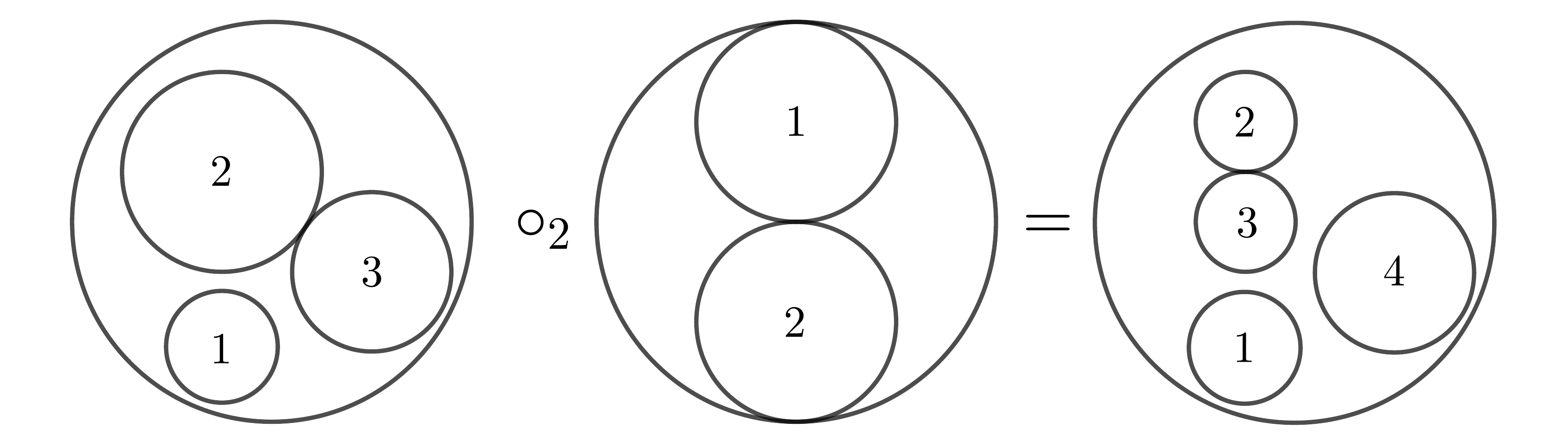}\vspace{-10pt}
\caption{Illustration of the operadic composition ${\circ}_{2}:\lD_{2}(3)\times \lD_{2}(2)\rightarrow \lD_{2}(4)$.}\label{B5}\vspace{-10pt}
\end{center}
\end{figure}
\end{ex}

\subsection{The Reedy model category of reduced bimodules over a reduced operad} \label{Final1}

Let $\op O$ be a reduced operad. A reduced bimodule $\op M$ over the operad $\op O$, or $\op O$-bimodule, is a $\Lambda$-sequence $\op M:\Lambda^{op}\rightarrow \TopCat$ together with operations of the form
\begin{equation}\label{B6}
\begin{array}{ll}\vspace{3pt}
{\circ}^{i}:\op M(n)\times \op O(m) \longrightarrow \op M(n+m-1), & \text{called right operations with } 1\leq i\leq n,\,\, m\geq 1 \\ 
\gamma:\op O(n)\times \op M(m_{1})\times \cdots \times M(m_{n})\longrightarrow \op M(m_{1}+\cdots + m_{n}) & \text{called left operation,}\,\, n\geq 1,
\text{ each $m_i\geq  1$},
\end{array} 
\end{equation}
satisfying some compatibility relations with the $\Lambda$-structure, associativity and unit axioms \cite{DucT}.  A map between $\op O$-bimodules should respect these operations. We denote by $\Bimod_{\op O}$ the category of reduced bimodules over the reduced operad $\op O$. In what follows, we use the notation
$$
\begin{array}{cl}
x\circ^{i}y=\circ^{i}(x,y) &   \text{for } x\in \op M(n) \text{ and } y\in \op O(m), \\ 
x(y_{1},\ldots,y_{n})=\gamma(x,y_{1},\ldots,y_{n}) & \text{for } x\in \op O(n) \text{ and } y_{i}\in \op M(m_{i}).
\end{array} 
$$

Given an integer $k\geq 1$, we also consider the category of $k$-truncated reduced bimodules over $\op O$ denoted by $\Bimod_{O\,;\,\leq k}$. The objects are $k$-truncated $\Lambda$-sequences together with operations of the form (\ref{B6}) with $n+m-1\leq k$ for the right operations and $m_{1}+\cdots + m_{n}\leq k$ for the left operation. Furthermore, one has the functor 
$$
\begin{array}{rcl}\vspace{3pt}
(-)_{\leq k}:\Bimod_{\op O} & \longrightarrow & \Bimod_{\op O\,;\,\leq k}; \\ 
 \op M & \longmapsto & \op M_{\leq k}:=\left\{ 
 \begin{array}{ll}\vspace{3pt}
 \op M_{\leq k}(n)=\op M(n) & \text{if } n\leq k, \\ 
 \op M_{\leq k}(n)=\emptyset & \text{if } n > k.
 \end{array} 
 \right.
\end{array} 
$$

For $k\geq 1$, the categories $\Bimod_{O}$ and $\Bimod_{O\,;\,\leq k}$ of reduced bimodules and $k$-truncated reduced bimodules over a reduced operad $O$, respectively, are also endowed with Reedy model category structures transferred from $\Lambda Seq$ and $\Lambda Seq_{\leq k}$, respectively, along the adjunctions 
$$
\begin{array}{rcl}\vspace{3pt}
\mathcal{F}_{B}: \Lambda Seq & \leftrightarrows & \Bimod_{O}:\mathcal{U}, \\ 
\mathcal{F}_{B\,;\,\leq k}: \Lambda Seq_{\leq k} & \leftrightarrows & \Bimod_{O\,;\,\leq k}:\mathcal{U},
\end{array} 
$$ 
where $\mathcal{U}$ is the forgetful functor while $\mathcal{F}_{B}$ and $\mathcal{F}_{B\,;\,\leq k}$ are the free bimodule functors. In other words, a map $f:\op P \rightarrow \op Q$ of (possibly truncated) $O$-bimodules is a weak equivalence (resp. a fibration) if the corresponding map $\mathcal{U}(f)$ is a weak equivalence (resp. a fibration) in the model category of (possibly truncated) $\Lambda$-sequences.

\begin{ex}
Let $\eta:\op P\rightarrow \op Q$ be a  map of operads. In that case, the map $\eta$ is also  a bimodule map over $\op P$ and the right operations of the bimodule structure on $\op Q$ are given by 
$$
\begin{array}{lrll}\vspace{3pt}
{\circ}^{i}: & \op Q(n)\times \op P(m)  & \longrightarrow &  \op Q(n+m-1); \\\vspace{7pt}
& (x\,;\,y) &\longmapsto & x{\circ}_{i}\eta(y),
\end{array}
$$
while the left operation is defined as follows:
$$
\begin{array}{lrll}\vspace{3pt}
\gamma: &  \op P(n)\times \op Q(m_{1})\times \cdots\times \op Q(m_{n})  & \longrightarrow &  \op Q(m_{1}+\cdots +m_{n});\\
& (x\,;\,y_{1},\ldots,y_{n})&\longmapsto & (\cdots(\eta(x){\circ}_{n}y_{n})\cdots){\circ}_{1}y_{1}.
\end{array}
$$
\end{ex}

\begin{ex}\label{B7}\textit{The fiber bundle bimodule $\op Q{\circ} X$.}

\noindent Let $f:\op P \rightarrow \op Q$ be a map of reduced operads and $(X\,;\,\ast)$ be a pointed space equipped with a map $\delta:X\rightarrow \Operad(\op P, \op Q)$ sending the basepoint to $f$. By convention, we denote by $\delta_{x}:\op P \rightarrow \op Q$ the operadic map associated to $x\in X$. We can think of $X$ as a space of $P$-bimodule structures on $\op Q$ by twisting the right module structure. Then, the $\Lambda$-sequence  $\op Q{\circ} X$, see Definition~\ref{d:circle}, 
 is a $\op P$-bimodule. 
The left operation is obtained using the operadic map $f$
$$
\begin{array}{rcl}\vspace{3pt}
\gamma:\op P(n)\times \op Q{\circ} X(m_{1})\times \cdots \times \op Q{\circ} X(m_{n}) & \longmapsto & \op Q{\circ} X(m_{1}+\cdots + m_{n});\\
p\,;\,\{(q_{i},x_{1}^{i},\ldots,x_{m_{i}}^{i})\} & \longmapsto & (\delta_{\ast}(p)(q_{1},\ldots,q_{n}),x_{1}^{1},\ldots,x_{m_{n}}^{n}),
\end{array}
$$
while the right operations are given by 
$$
\begin{array}{rcl}\vspace{3pt}
{\circ}^{i}: \op Q{\circ} X(n)\times \op P(m) & \longrightarrow & \op Q{\circ} X(n+m-1); \\ \vspace{5pt}
(q,x_{1},\ldots,x_{n})\,;\,p & \longmapsto & (q{\circ}_{i}\delta_{x_{i}}(p),x_{1},\ldots,\underset{m \text{ times}}{\underbrace{x_{i},\ldots,x_{i}}},\ldots, x_{n}).
\end{array}
$$

In the sequel we  consider a more general situation -- when the map $\delta$ sends $X$ to the
derived operadic mapping space:
\beq{eq:mapX}
\delta\colon X\to \Operad^h(\op P,\op Q).
\eeq
Assuming $Q$ is Reedy fibrant (if necessary by taking its fibrant replacement $\op Q\to \op Q^f$) and $\op P$ 
is  $\Sigma$-cofibrant, the
target of $\delta$ is the operadic mapping space $\Operad^h(\op P,\op Q)=\Operad(W\op P,Q)$,
where $W\op P$ is the Boardman-Vogt resolution of $\op P$ reviewed in the next section. Thus $\op Q
{\circ} X$ is given the structure of a $W\op P$-bimodule. Let us mention, however, that the restriction-induction adjunction
\beq{eq:restr_ind}
\mathrm{Ind}\colon \Bimod_{W\op P}\rightleftarrows\Bimod_{\op P}\colon \mathrm{Restr}
\eeq
is a Quillen equivalence~\cite[Theorem 3.7]{DucTF}, and therefore it does not matter which of the two homotopy categories of bimodules we consider.
\end{ex}
\vspace{1pt}

{\bf Characterisation of cofibrations.} Let $\op P$ be a reduced $\Sigma$-cofibrant operad. Denote by $\Sigma_{>0}\Bimod_{\op P}$ the category 
of \lq\lq{}usual\rq\rq{} bimodules $M=\{M(n),\, n\geq 1\}$ over $\op P$ which  means: its objects $M$ have   a componentwise
symmetric group action and a $\op P$-action~\eqref{B6}, but they are not given  the $\Lambda$ structure. This category is endowed with the {\it projective} model structure for which
weak equivalences and fibrations are those componentwise~\cite{DucTF}. By \cite[Theorem~3.5]{DucTF}, a morphism $M\to N$ in the
Reedy model category $\Bimod_{\op P}$ of reduced bimodules is a cofibration if and only if it is one in the projective model category $\Sigma_{>0}\Bimod_{\op P}$.

\section{Delooping derived mapping spaces of bimodules}

\subsection{The Boardman-Vogt resolution in the category of operads}\label{sec:ducmap}
Let $\op P$ be a reduced operad. We denote its Boardman-Vogt resolution by $W\op P$. Its points are equivalence classes $[T\,;\,\{t_{e}\}\,;\,\{a_{v}\}]$ where $T$ is a rooted tree, $\{a_{v}\}_{v\in V(T)}$ is a family of points in $\op P$ labelling the vertices of $T$ and $\{t_{e}\}_{e\in E^{int}(T)}$ is a family of real numbers in the interval $[0\,,\,1]$ indexing the inner edges. In other words, one has 
$$
W\op P(n):=
\left.
\underset{T\in \,\text{\textbf{tree}}_{n}}{\displaystyle\coprod} \,\,\underset{v\in V(T)}{\displaystyle\prod}\,\op P(|v|) \,\,\times \,\,\underset{e\in\, E^{int}(T)}{\displaystyle\prod}\, [0\,,\,1]\,\,
\right/\!\sim\,\,\hspace{20pt} \text{with } n\geq 1,
$$
where $\text{\textbf{tree}}_{n}$ is the set of planar rooted trees without univalent vertices and with $n$ leaves labelled by the set $\{1,\ldots,n\}$. In other words, such a decoration of the leaves can be considered as an element in permutation group $\Sigma_{n}$ thanks to the planar order. The equivalence relation is generated by the unit axiom (i.e. we remove vertices labelled by the unit of the operad $\op P$) and the compatibility with the symmetric group axiom (a vertex $v$ labelled by $a\cdot\sigma$, with $\sigma\in \Sigma_{|v|}$, is identified with $a$ by permuting the incoming edges of $v$ according to $\sigma$). Furthermore, if an inner edge is indexed by $0$, then we contract it by using the operadic structure of $\op P$.
\begin{figure}[!h]
\begin{center}
\includegraphics[scale=0.35]{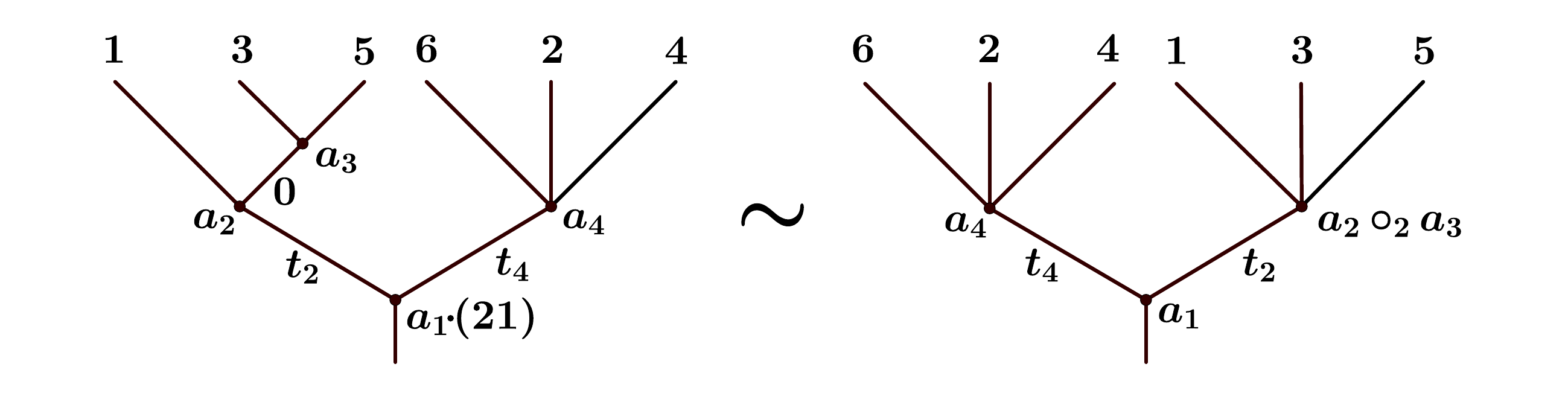}\vspace{-5pt}
\caption{Illustration of the equivalence relation.}
\end{center}
\end{figure}

Let $[T\,;\,\{t_{e}\}\,;\,\{a_{v}\}]$ be a point in $W\op P(n)$ and $[T'\,;\,\{t'_{e}\}\,;\,\{a'_{v}\}]$ be a point in $W\op P(m)$. The operadic composition $[T\,;\,\{t_{e}\}\,;\,\{a_{v}\}]{\circ}_{i}[T'\,;\,\{t'_{e}\}\,;\,\{a'_{v}\}]$ is obtained by grafting $T'$ to the $i$-th incoming input of $T$ and indexing the new inner edge by $1$. The $\Lambda$-structure  is defined by permuting the leaves and contracting some of them using the $\Lambda$-structure of the operad $\op P$. Furthermore, there is a map of operads sending the real numbers indexing the inner edges to $0$
\begin{equation}\label{d8}
\mu:W\op P\rightarrow \op P\,\,;\,\, [T\,;\,\{t_{e}\}\,;\,\{a_{v}\}] \mapsto [T\,;\,\{0_{e}\}\,;\,\{a_{v}\}].
\end{equation}

\begin{thm}\label{I0}{\cite[Theorem 5.1]{BM}, \cite[Theorem II.8.4.12]{Fr}}
Assume that $\op P$ is a  $\Sigma$-cofibrant operad. The objects $W\op P$ and $(W\op P)_{\leq k}$ are cofibrant replacements of $\op P$ and $\op P_{\leq k}$ in the categories $\Operad$ and $\Operad_{\leq k}$, respectively. In particular, the map (\ref{d8}) is a weak equivalence. 
\end{thm}

From now on, we introduce a filtration of the resolution $W\op P$ according to the arity. A point in $W\op P$ is said to be \textit{prime} if the real numbers indexing the set of inner edges are strictly smaller than $1$. Besides, a point is said to be \textit{composite} if one of its inner edges is indexed by $1$ and such a point can be decomposed into prime components. More precisely, the prime components of a point indexed by a tree are obtained by cutting the edges labelled by $1$. \vspace{7pt}

A prime point is in the $k$-th filtration term $W\op P_{k}$ if it has  at most $k$ leaves. Then, a composite point is in the $k$-th filtration term if its prime components are in $W\op P_{k}$. For instance, the composite point in Figure \ref{G2} is an element in the filtration term $W\op P_{4}$. By convention, $W\op P_{0}$ is the initial object in the category of operads. For each $k\geq 0$, $W\op P_{k}$ is a reduced operad and the family $\{W\op P_{k}\}$ produces the following filtration of $W\op P$: 
\begin{equation*}\label{G4}
\xymatrix{
W\op P_{0}\ar[r] & W\op P_{1} \ar[r] &  \cdots \ar[r] & W\op P_{k-1} \ar[r] & W\op P_{k} \ar[r] & \cdots \ar[r] & W\op P.
}\vspace{3pt}
\end{equation*}  

\begin{figure}[!h]
\begin{center}
\includegraphics[scale=0.35]{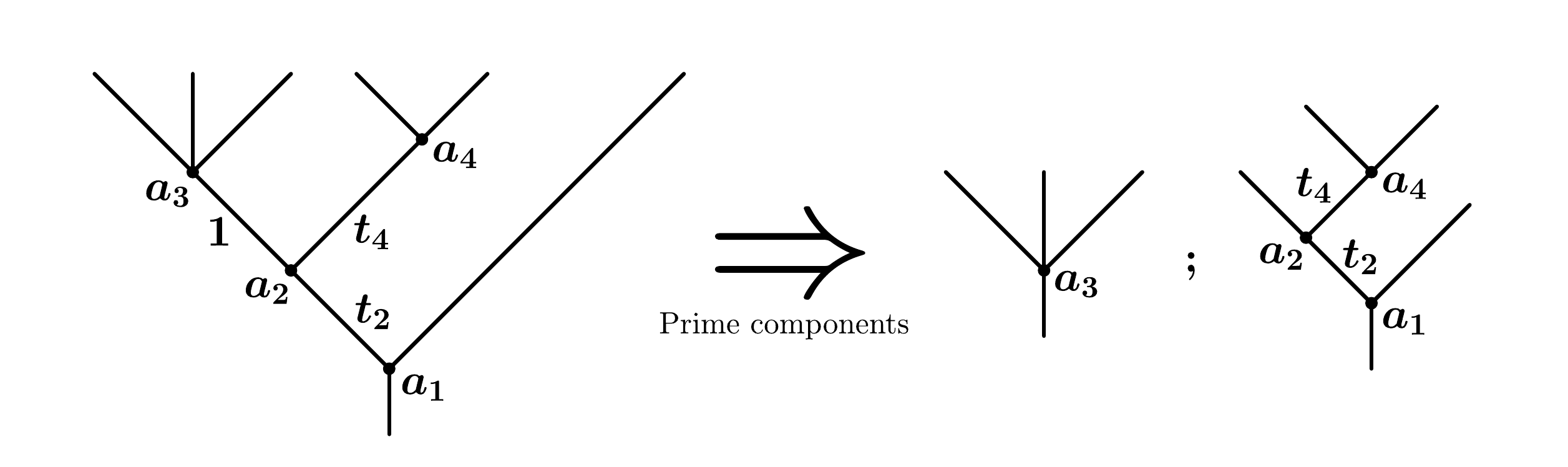}\vspace{-10pt}
\caption{Illustration of a composite point together with its prime components.}\label{G2}
\end{center}
\end{figure}

From a $k$-truncated reduced operad $\op P_{k}$, we consider the $k$-free operad $\mathcal{F}_{Op_{k}}(\op P_{k})$ whose $k$ first components coincide with $\op P_{k}$. The functor $\mathcal{F}_{Op_{k}}$ is left adjoint to the truncation functor $(-)_{\leq k}$ and it can be expressed as a quotient of the free operad functor in which the equivalence relation is generated by the following axiom: any composite 
element is identified with the composition of its prime components.
   In our case, we can easily check that $\mathcal{F}_{Op_{k}}((W \op P)_{\leq k})=W\op P_{k}$, since $W\op P_{k}$ is the sub-operad of $W\op P$ generated by its $k$ first components. Consequently, from this adjunction and Theorem~\ref{I0}, we deduce the following identifications:
\begin{equation}\label{eq:tr_op_map}
\Operad_{\leq k}((W \op P)_{\leq k}\,,\,Q_{\leq k}) \cong  \Operad(W\op P_{k}\,,\,Q).
\end{equation}

\subsection{A cofibrant resolution of $\op P$ in the category of bimodules over itself}

The operad $\op P$ may naturally be considered as a reduced bimodule of itself. We will use (a slight variant of) the cofibrant resolution $B\op P$ of $\op P$ as a bimodule introduced by Ducoulombier in \cite{Duc}. The points are equivalence classes $[T\,;\,\{t_{v}\}\,;\,\{x_{v}\}]$ where $T$ is a tree, $\{t_{v}\}$ is a family of real numbers in the interval $[0\,,\,1]$ indexing the vertices and $\{x_{v}\}$ is a family of points in $W\op P$ labelling the vertices. Furthermore, if $e$ is an inner edge of $T$, then the real numbers $t_{s(e)}$ and $t_{t(e)}$ indexing respectively the source and the target vertices of $e$ according to the orientation toward the root satisfy the relation $t_{s(e)}\geq t_{t(e)}$:
\begin{equation}\label{B9}
B\op P (n)\subset 
\left.
\underset{T\in \textbf{tree}_{n}}{\displaystyle\coprod}\,\,\,\underset{v\in V(T)}{\displaystyle\prod}\, W\op P(|v|)\times [0\,,\,1]\,
\right/ \sim,\hspace{20pt}\text{with } n\geq 1.
\end{equation}
The equivalence relation is generated by the unit (i.e. we can remove the vertices indexing by the unit $\ast_{1}$ the operad $W\op P$) and the compatibility with the symmetric group axioms. Furthermore, if two vertices joined by an edge have the same height, then the edge may be contracted, using the operadic composition in $W\op P$ as illustrated in Figure~\ref{Fig4}.

\begin{figure}[!h]
\begin{center}
\includegraphics[scale=0.12]{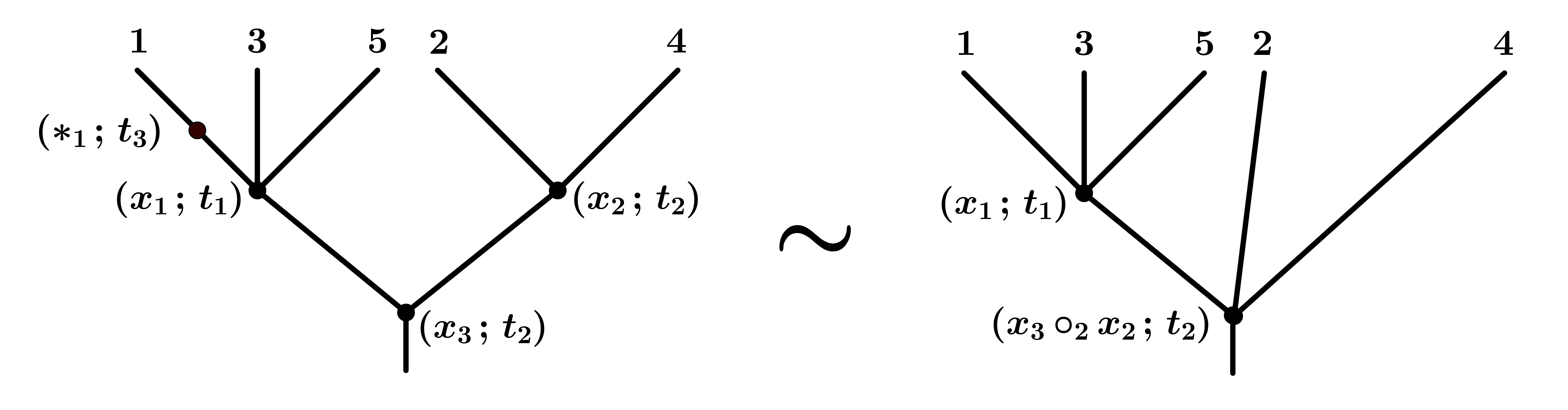}
\caption{Illustration of the equivalence relation on $B\op P(5)$.}\label{Fig4}
\end{center}
\end{figure}

The object so obtained inherits a bimodule structure over $W\op P$. The left and right module structures along a point in $W\op P(m)$, with $m\geq 1$, are both obtained  by grafting trees, with the newly formed vertices being assigned height 0 for the left module structure and height 1 for the right module structure. Moreover, the $\Lambda$-structure is defined by permuting some leaves and contracting the other ones using the $\Lambda$ structure of $W\op P$. Furthermore, there is a map of bimodules sending the real numbers indexing the vertices to $0$:
\begin{equation}\label{B1}
\mu':B\op P\rightarrow W\op P\,\,;\,\, [T\,;\,\{t_{v}\}\,;\,\{x_{v}\}] \mapsto \mu([T\,;\,\{0_{v}\}\,;\,\{x_{v}\}]).\vspace{-17pt}
\end{equation}
\begin{figure}[!h]
\begin{center}
\includegraphics[scale=0.35]{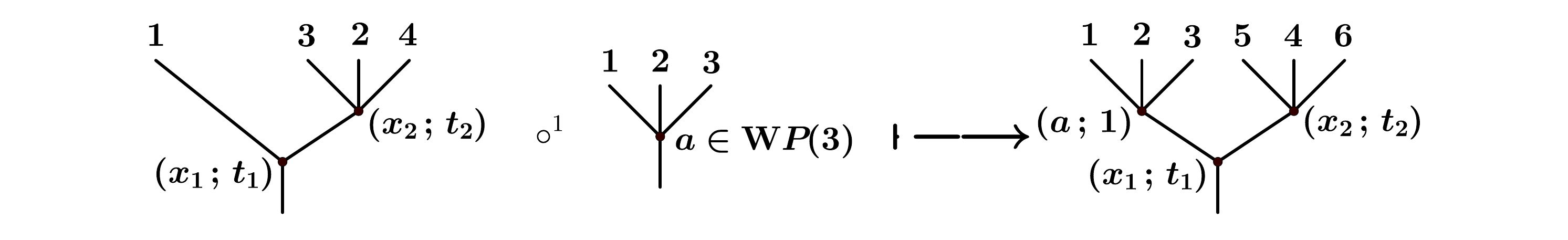}
\caption{Illustration of the right operation ${\circ}^{1}:B\op P(4)\times W\op P(3)\rightarrow B\op P(6)$.}\label{B2}
\end{center}
\end{figure}

\begin{thm}{\cite[Theorem 2.6]{Duc},\cite[Proposition 3.9]{DucTF}}\label{D8}
Assume that $\op P$ is a  $\Sigma$-cofibrant operad. Then, the objects $B\op P$ and $(B\op P)_{\leq k}$ are cofibrant replacements of $\op P$ and $\op P_{\leq k}$ in the categories $\Bimod_{W\op P}$ and $\Bimod_{W\op P\,;\,\leq k}$, respectively. In particular, the map (\ref{B1}) is a weak equivalence.
\end{thm}\vspace{-5pt}
 
\begin{figure}[!h]
\begin{center}
\includegraphics[scale=0.5]{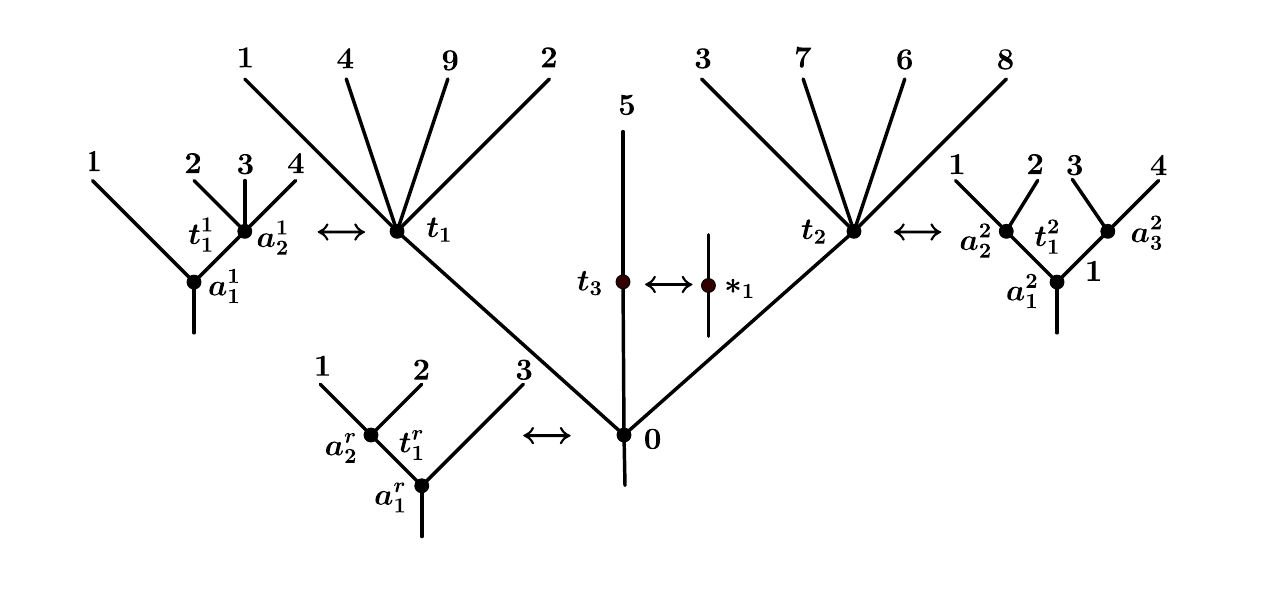}\vspace{-13pt}
\caption{An alternative representation of a point in $B\op P(9)$.}\label{C2}
\end{center}
\end{figure}

From now on, we introduce a filtration of the resolution $B\op P$ according to the arity. 
 Similarly to the operadic case 
a point in $B\op P$ is said to be \textit{prime} if the real numbers indexing the vertices of the main tree are in the interval $]0\,,\,1[$. Besides, a point is said to be \textit{composite} if one vertex of the main tree is indexed by $0$ or $1$ and such a point can be decomposed into prime components. More precisely, thanks to the unit axiom, we can always choose a representative point for which each path joining a leaf to the root passes through at least one vertex not indexed by $0$ or $1$ (by adding a bivalent vertex labelled by the unit if necessary). The prime components of such a point are obtained by removing the vertices of the main tree indexed by $0$ or $1$. For instance, the three prime components associated to the composite point in Figure \ref{C2} are the following ones: \vspace{15pt}

\hspace{-40pt}\includegraphics[scale=0.21]{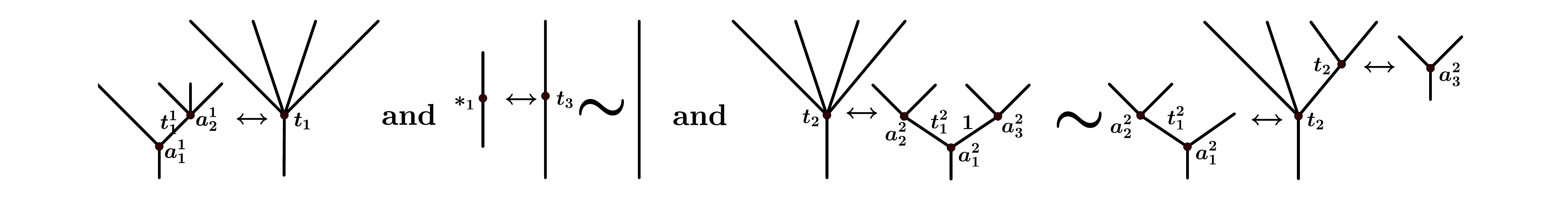}\vspace{15pt}

A prime point is in the $k$-th filtration term $B\op P_{k}$ if it has at most $k$ leaves. Similarly, a composite point is in the $k$-filtration if its prime components are in $B\op P_{k}$. For instance, the composite point in Figure \ref{C2} is an element in the filtration term $B\op P_{4}$. In particular, if the vertices of a point in $B\op P$ are indexed only by $0$ or $1$, then the point is in the first filtration term since its prime components are $1$-corollas indexed by the unit of the operad $W\op P$ (this element can also be considered as the vertexless tree ``$|$''). By convention, $B\op P_{0}$ is the initial element in the category of bimodules over $W\op P$ -- it is empty in all arities~$\geq 1$. The family $\{B\op P_{k}\}$ produces the following filtration of $B\op P$:\vspace{2pt}
\begin{equation}\label{C1}
\xymatrix{
B\op P_{0}\ar[r] & B\op P_{1} \ar[r] &  \cdots \ar[r] & B\op P_{k-1} \ar[r] & B\op P_{k} \ar[r] & \cdots \ar[r] & B\op P_\infty:= B\op P.
}
\end{equation}

It is shown by \cite[Theorem 2.6]{Duc} and \cite[Proposition 3.9]{DucTF} that each inclusion above is a Reedy cofibration
(in $\Bimod_{W\op P}$) and  in any arity $n$, the inclusion $P_{k-1}(n)\to P_k(n)$ is a $\Sigma_n$-cofibration.

Analogously to the operadic case, from a $k$-truncated bimodule $M_{k}$, we consider the $k$-free bimodule $\mathcal{F}_{B_{k}}(M_{k})$ whose $k$ first components coincide with $M_{k}$. Similarly to the operadic case, the functor $\mathcal{F}_{B_{k}}$ is left adjoint to the truncation functor $(-)_{\leq k}$ and can be expressed as a quotient of the free bimodule functor in which the equivalence relation is generated by the following axiom: any composite element in $M_k$ is equivalent to the corresponding product of its prime components.  
 One has, $\mathcal{F}_{B_k}(B\op P_{\leq k})=B\op P_k$. 
 Consequently, there are the following identifications:
\begin{equation}\label{eq:tr_bim_map}
\Bimod_{W\op P\,;\,\leq k}((B\op P)_{\leq k}\,;\,Q_{\leq k}) \cong \Bimod_{W\op P}(B\op P_{k}\,;\,Q).
\end{equation}

\subsection{The weak equivalence of $\lD_{1}$-algebras}
In the previous subsection we introduced a cofibrant replacement $B\op P$ of an operad $\op P$ in the category of bimodules over $W\op P$. In \cite{Duc}, the author uses this resolution in order to equip the corresponding model of the derived mapping space of bimodules with a structure of $\lD_{1}$-algebra. Then, he shows the following statement:

\begin{thm}{\cite[Theorem 3.1]{Duc}}\label{thm:duc}
Let $P$ be a  $\Sigma$-cofibrant operad and $\eta:P\rightarrow Q$ be a map of reduced operads. If the space $Q(1)$ is weakly contractible, then there are explicit weak equivalences of $\lD_{1}$-algebras:
\begin{equation}\label{Mapcase1}
\begin{array}{rcl}\vspace{5pt}
\xi: \Omega \Operad^{h}(P\,;\,Q) & \longrightarrow & \Bimod_{W\op P}^{h}(P\,;\,Q) , \\ 
\xi_{k}:\Omega \big( \Operad^{h}_{\leq k}(P_{\leq k}\,;\,Q_{\leq k})\big) & \longrightarrow & \Bimod_{W\op P\,;\,\leq k}^{h}(P_{\leq k}\,;\,Q_{\leq k}) .
\end{array} 
\end{equation}
\end{thm}

In what follows, we assume that the operad $\op Q$ is fibrant in the Reedy model category of reduced operads. If it is not the case, then we substitute $\op Q$ with any fibrant resolution $\op Q^{f}$. Such resolution is equipped with a map $\tilde{\eta}:\op P\rightarrow \op Q \rightarrow \op Q^{f}$ making $Q^{f}$ into a fibrant object in both categories of reduced operads and bimodules over $W\op P$. Similarly, for any $k\geq 1$, the $k$-truncated operad $\op Q^{f}_{\leq k}$ gives rise to a fibrant replacement of $\op Q_{\leq k}$ in the categories of $k$-truncated operads and $k$-truncated bimodules over $W\op P$.

By using the resolutions $W\op P$ and $W\op P_{k}$ for (truncated) operads as well as the resolutions $B\op P$ and $B\op P_{k}$ for (truncated) bimodules, we can easily define the maps $\xi$ and $\xi_{k}$. First of all, we recall that a point in the loop space $\Omega \Operad(W\op P\,;\,Q)$, based in $\eta{\circ}\mu:W\op P \rightarrow \op P\rightarrow Q$, is given by a family of maps\vspace{2pt}
$$
g_{n}:W\op P(n)\times [0\,,\,1]\longrightarrow Q(n),\hspace{15pt}\forall n\geq 1,\vspace{2pt}
$$
satisfying the following conditions:\vspace{5pt}
\begin{itemize}
\item[$\blacktriangleright$] $g_{n}(\iota(\ast_{1})\,;\, t)=\ast_{1}$\hspace{93pt} $\forall t\in [0\,,\,1]$,\vspace{5pt}
\item[$\blacktriangleright$] $g_{n}(x{\circ}_{i}y\,;\, t)=g_{l}(x\,;\, t){\circ}_{i}g_{n-l+1}(y\,;\, t)$\hspace{20pt} $\forall t\in [0\,,\,1]$, $x\in W\op P(l)$ and $y\in W\op P(n-l+1)$,\vspace{5pt}
\item[$\blacktriangleright$] $g_{n}(x\,;\, t)=\eta{\circ} \mu(x)$\hspace{90pt} $t\in \{0\,;\,1\}$ and $x\in W\op P(n)$,\vspace{5pt}
\item[$\blacktriangleright$] $g_{m}(u^{\ast}(x)\,;\, t)=u^{\ast}(g_{n}(x\,;\, t))$\hspace{50pt} $\forall t\in [0\,,\,1]$, $x\in W\op P(n)$ and $u\in \Lambda([m]\,,\,[n])$.
\end{itemize}\vspace{5pt}

Let $g=\{g_{n}\}$ be a point in the loop space and let $[T\,;\,\{t_{v}\}\,;\,\{x_{v}\}]$ be a point in $B\op P$.
 This element is a tree whose vertices are labelled by pairs $(x_{v},t_{v})$. To obtain $\xi(g)\bigl(
 [T\,;\,\{t_{v}\}\,;\,\{x_{v}\}]\bigr)$ we replace the label  of each  vertex $v$ of $T$ 
 by $g_{|v|}(x_{v},t_v)\in \op Q(|v|)$ and then we compose all these elements using the structure of $T$ and 
 the operadic compositions of $\op Q$.
  For instance, the image of the point $[T\,;\,\{t_{v}\}\,;\,\{x_{v}\}]\in B\op P(6)$ associated to the operadic composition in Figure \ref{B2} is the following one:\vspace{3pt}
$$
\begin{array}{cl}\vspace{5pt}
\xi(g)([T\,;\,\{t_{v}\}\,;\,\{x_{v}\}]) & = g_{2}(x_{1}\,;\,t_{1})\big( g_{3}(a\,;\, 1)\,;\,g_{3}(x_{2}\,;\,t_{2})\big),  \\ 
 & =g_{3}(x_{1}\,;\,t_{1})\big( \mu{\circ}\eta(a)\,;\,g_{3}(x_{2}\,;\,t_{2})\big).
\end{array} \vspace{3pt}
$$

\section{The homotopy fiber case}

For the rest of this section, $\eta:\op P \rightarrow \op Q$ is a map of reduced operads and $(X\,;\,\ast)$ is a pointed space equipped  with a map $\delta:X\rightarrow \Operad(W\op P, Q)$ sending the basepoint to the composite map $\eta{\circ} \mu:W\op P \rightarrow P \rightarrow Q$. According to the notation introduced in Example \ref{B7}, applied to the composite map $\eta{\circ}\mu$, one has a $W\op P$-bimodule $\op Q{\circ} X$. 
  The purpose of this section is to prove the following theorem:

\begin{thm}\label{th:main}
Suppose that $X$ is a path-connected pointed space; $\op P$ and $\op Q$ are reduced topological operads; $\op P$ is   $\Sigma$-cofibrant; $\op Q$ is Reedy fibrant; $\op P(1)$ and $\op Q(1)$ are weakly contractible. Then, the following are homotopy fiber sequences:
\begin{equation}\label{Mapcase2}
\begin{array}{c}\vspace{3pt}
\Bimod_{W\op P}(B\op P, \op Q{\circ} X) \longrightarrow X \longrightarrow \Operad(W\op P, \op Q), \\ 
\Bimod_{W\op P}(B\op P_{k}, \op Q{\circ} X) \longrightarrow X \longrightarrow \Operad(W\op P_{k}, \op Q).
\end{array} 
\end{equation}
\end{thm}

\begin{proof}
We only prove the statement in the usual case. The same arguments work for the truncated case.
The result is a consequence of Theorems \ref{THmPart1} and \ref{ThmPart2} in which we introduce an intermediate space $\Bimod_{W\op P\,;\,X}(B\op P\,;\,\op Q)$ together with explicit weak equivalences
$$
\xymatrix{
\psi:hofiber\big( X\rightarrow \Operad(W\op P,\op Q) \big) \ar[r]^{\hspace{25pt}\psi'} & \Bimod_{W\op P\,;\,X}(B\op P\,;\,\op Q) \ar[r]^{\hspace{-10pt}\psi''} & \Bimod_{W\op P}(B\op P\,;\,\op Q{\circ} X).
}
$$ 

\end{proof}

For the rest of the section we will be assuming that $\op P$ and $\op Q$ are reduced topological operads;
$\op P$ is  $\Sigma$-cofibrant; $\op Q$ is Reedy fibrant.

\subsection{A bundle of bimodule maps}\label{sec:bimod_bundle}

For $x\in X$ we denote by $\op Q_x$ the $W\op P$-bimodule  obtained from $\op Q$ by using the map $\delta_{x}$ to define the right $W\op P$-action and the map $\delta_{\ast}$ to define the left $W\op P$-action. In other words, the $W\op P$-bimodule structure of $\op Q_x:=\{Q_{x}(n)=Q(n),\, n\geq 1\}$ is given by the following formulas: 
$$
\begin{array}{rcl}\vspace{3pt}
{\circ}^{i}: \op Q_{x}(n)\times W\op P(m) & \longrightarrow & \op Q_{x}(n+m-1); \\ \vspace{5pt}
q\,,\,p & \longmapsto & q{\circ}_{i}\delta_{x}(p),\\\vspace{2pt}
\gamma:W\op P(n)\times \op Q_{x}(m_{1})\times \cdots \times \op Q_{x}(m_{n}) & \longmapsto & \op Q_{x}(m_{1}+\cdots + m_{n});\\
p,q_{1},\ldots,q_{n} & \longmapsto & \delta_{\ast}(x)(q_{1},\ldots,q_{n}).
\end{array} 
$$

Next, for a reduced $W\op P$-bimodule $M$, define $\Bimod_{W\op P,X}(M,\op Q )$ to be the space consisting of pairs $(x, f)$, where $x\in X$ and $f\in \Bimod_{W\op P}(M,\op Q_x)$.
There is a natural inclusion 
\[
\psi'':\Bimod_{W\op P,X}(B \op P,\op Q ) \longrightarrow \Bimod_{W\op P}(B \op P,\op Q{\circ} X )
\]
such that the image of $(x,f)$ as above is the map
\[
B\op P \ni p\longmapsto (f(p),x,\dots, x).
\]

\begin{prop}\label{lem:fib}
The truncations 
\begin{equation}\label{truncMap1}
\Bimod_{W\op P}(B\op P_k,\op Q{\circ} X) \longrightarrow \Bimod_{W\op P}(B\op P_{\ell},\op Q{\circ} X)
\end{equation}
and
\begin{equation}\label{truncMap2}
\Bimod_{W\op P,X}(B \op P_k,\op Q ) \longrightarrow \Bimod_{W\op P,X}(B\op P_{\ell},\op Q)
\end{equation}
are Serre fibrations for any $0\leq \ell\leq k\leq\infty$.
\end{prop}

\begin{proof}
 It is enough to prove it for $\ell=k-1<\infty$. 
Let us assume first that $\op P(1)=*$. We introduce the subspace $\partial B \op P(k)$, consisting of points in $B \op P(k)$ that have at least one vertex of the main tree labelled by~$0$ or~$1$.
  Since we assume $\op P(1)=*$, one has $\partial B \op P(k)=B\op P_{k-1}(k)$.  The space $\partial B \op P(k)$ is equipped with an action of the symmetric group $\Sigma_{k}$ and the inclusion $\partial B\op P(k)\to B\op P(k)$ is a $\Sigma_k$-cofibration~\cite{Duc}.
%
 The inclusion of bimodules $B\op P_{k-1}\to B\op P_k$ is a Reedy cofibration because it fits in the following pushout diagram
 in $\Sigma_{>0}\Bimod_{W\op P}$, see {\bf Characterisation of cofibrations} in Subsection~\ref{Final1}: 
\begin{equation}\label{eq:sigma_pushout1}
\xymatrix{
\mathcal{F}_{W\op P}^\Sigma(\partial B\op P(k))\ar[r] \ar[d] & \mathcal{F}_{W\op P}^\Sigma(B\op P(k)) \ar[d] \\
B \op P_{k-1} \ar[r] & B \op P_{k},
}
\end{equation}
where $\mathcal{F}_{W\op P}^\Sigma(\partial B\op P(k))$ and $\mathcal{F}_{W\op P}^\Sigma(B\op P(k))$ denote the free
bimodules in $\Sigma_{>0}\Bimod_{W\op P}$ generated by the $\Sigma$-sequences $\partial B\op P(k)$ and $B\op P(k)$
concentrated in the only arity~$k$. Therefore, given a reduced bimodule map $B\op P_{k-1}\xrightarrow{\lambda_{k-1}}Q{\circ} X$,
in order to extend it to  $B\op P_k\xrightarrow{\lambda_k}Q{\circ} X$, one has to define a $\Sigma_k$-equivariant 
map $B\op P(k)\xrightarrow{\lambda_k(k)}Q{\circ} X(k)$, such that $\lambda_k(k)|_{\partial B\op P(k)}=\lambda_{k-1}(k)$
and which respects the $\Lambda$ structure. The latter condition is equivalent to the commutativity of the square
$$
\xymatrix{
B\op P(k)\ar[rr]^{\lambda_k(k)} \ar[d] &{}& Q{\circ} X(k) \ar[d] \\
\mathcal{M}(B\op P)(k)=\mathcal{M}(B \op P_{k-1})(k) \ar[rr]^-{\mathcal{M}(\lambda_{k-1})(k)} &{}& \mathcal{M}(Q{\circ} X)(k).
}
$$
Consequently, one has the pullback diagram
\begin{equation}\label{eq:pullback1}
\xymatrix{
\Bimod_{W\op P}(B \op P_{k},\op Q {\circ} X)\ar[d] \ar[r] & \TopCat_{\Sigma_{k}}( B \op P(k)\,,\, \op Q{\circ} X(k)) \ar[d] \\
\Bimod_{W\op P}(B\op P_{k-1},\op Q{\circ} X) \ar[r] & \TopCat_{\Sigma_{k}}(\partial  B \op P(k)\,,\,\op Q{\circ} X(k)) \underset{\TopCat_{\Sigma_{k}}(\partial  B \op P(k)\,,\,\mathcal{M}(\op Q{\circ} X)(k))}{\bigtimes} \TopCat_{\Sigma_{k}}(  B \op P(k)\,,\,\mathcal{M}(\op Q{\circ} X)(k)),
}
\end{equation}
where $\TopCat_{\Sigma_{k}}$ is the model category of spaces equipped with an action of the symmetric group $\Sigma_{k}$ and $\mathcal{M}(-)$ is the matching object (\ref{B3}).  Since the operad $\op Q$ is assumed to be fibrant in the Reedy model category of reduced operads, the same is true for the bimodule $\op Q{\circ} X$ due  to Lemma~\ref{l:circle}. 
 Furthermore,  
   the inclusion from $\partial B \op P(k)$ into $B \op P(k)$ is a $\Sigma_{k}$-cofibration. Consequently, the vertical maps in the above diagram are Serre fibrations by applying  an alternative version of the pushout product axiom along the inclusion $\partial B \op P(k)\rightarrow B \op P(k)$ and the map $\op Q{\circ} X(k)\rightarrow \mathcal{M}(\op Q{\circ} X)(k)$.
  
  In case $\op P(1)\neq *$ one has to consider an auxiliary filtration in the inclusion $B\op P_{k-1}
  \subset B\op P_k$, $k\geq 1$:
  \[
  B\op P_{k-1} =B\op P_{k-1,-1}\subset B\op P_{k-1,0}\subset B\op P_{k-1,1}\subset B\op P_{k-1,2}\subset
  \ldots \subset B\op P_{k},
  \]
  where $B\op P_{k-1,i}$ is a subbimodule of $B\op P_k$ generated by the prime components 
  of arity $\leq k-1$ and also of arity $k$ with $\leq i$ vertices. This finer filtration has the advantage to control the number of bivalent vertices and, in particular, to deal with the unit axiom. An argument similar to the one above
  shows that each map 
  \[
  \Bimod_{W\op P}(B\op P_{k-1,i},\op Q{\circ} X) \to \Bimod_{W\op P}(B\op P_{k-1,i-1},\op Q{\circ} X) 
  \]
   is a Serre fibration. Indeed, one similarly to~\eqref{eq:sigma_pushout1} has a pushout square in $\Sigma_{>0}\Bimod_{W\op P}$
   \begin{equation}\label{eq:sigma_pushout2}
\xymatrix{
\mathcal{F}_{W\op P}^\Sigma(B\op P_{k-1,i-1}(k))\ar[r] \ar[d] & \mathcal{F}_{W\op P}^\Sigma(B\op P_{k-1,i}(k)) \ar[d] \\
B \op P_{k-1,i-1} \ar[r] & B \op P_{k-1,i}.
}
\end{equation}
This implies that one has a pullback square obtained from~\eqref{eq:pullback1} by replacing  $B\op P_{k-1}$ and 
$B\op P_k$ in the left column by $B\op P_{k-1,i-1}$ and $B\op P_{k-1,i}$, respectively,  and replacing $\partial B\op P(k)$ and
$B\op P(k)$ in the right column by $B\op P_{k-1,i-1}(k)$ and $B\op P_{k-1,i}$, respectively. 

Similarly, we prove that the second truncation map (\ref{truncMap2}) is a Serre fibration. For example assuming $\op P(1)=*$ and using the above notation, one has the following pullback diagram in which the space $X$ does not appear in the right-hand terms since it has been fixed in the space $\Bimod_{W\op P,X}(B\op P_{k-1},\op Q)$:
$$
\xymatrix{
\Bimod_{W\op P,X}(B \op P_{k},\op Q)\ar[d] \ar[r] &  \TopCat_{\Sigma_{k}}( B \op P(k)\,,\, \op Q(k)) \ar[d] \\
\Bimod_{W\op P,X}(B\op P_{k-1},\op Q) \ar[r] & \TopCat_{\Sigma_{k}}(\partial  B \op P(k)\,,\,\op Q(k)) \underset{\TopCat_{\Sigma_{k}}(\partial  B \op P(k)\,,\,\mathcal{M}(\op Q)(k))}{\bigtimes} \TopCat_{\Sigma_{k}}(  B \op P(k)\,,\,\mathcal{M}(\op Q)(k))
}
$$
So the vertical maps of the above diagram and the truncation map (\ref{truncMap2}) are also Serre fibrations.
\end{proof}

\begin{lemma}\label{lem:T1}
Suppose that $\op P(1)$ is contractible.
Then the natural map
\begin{equation}\label{mapsec1}
 \Bimod_{W\op P,X}(B\op P_{1},\op Q) \longrightarrow  \Bimod_{W\op P}(B\op P_{1},\op Q{\circ} X)
\end{equation}
is a weak equivalence.
\end{lemma}

\begin{proof}
We shall in fact show that the arrows in the commutative diagram 
\[
\begin{tikzcd}
\Bimod_{W\op P,X}(B\op P_{1},\op Q)=\Bimod_{W \op P,X,\leq 1}(B\op P,\op Q) \ar{r}{\sim} \ar{d}{\sim} & \op Q(1) \times X \\
\Bimod_{W\op P}(B\op P_{1},\op Q{\circ} X)=\Bimod_{W\op P,\leq 1}(B\op P_{\leq 1},(\op Q{\circ} X)_{\leq 1}) \ar{ru}{\sim} & 
\end{tikzcd}
\]
are weak equivalences. Let $\mathbbm{1}$ denote the initial element in the category of reduced operads.
It is a point in arity one and empty in all the other arities $\geq 2$. Since $\op P(1)$ is contractible, the natural inclusion 
$\mathbbm{1}_{\leq 1} \to W\op P_{\leq 1}$ is a weak equivalence of 1-truncated reduced operads. As
a
consequence 
 for 1-truncated  $W\op P$-bimodules $\op M$, $\op M'$ the restriction map 
\[
\Bimod^h_{W\op P;\leq 1}(\op M, \op M') \to \Bimod^h_{\mathbbm{1};\leq 1}(\op M, \op M')
\]
is a weak equivalence. On the other hand, a reduced 1-truncated bimodule $\op M$ over~$\mathbbm{1}$ 
is just a space $\op M(1)$ with no additional structure. Thus provided $\op M(1)$ is a cofibrant space, 
\[
\Bimod^h_{\mathbbm{1};\leq 1}(\op M, \op M')=\Bimod_{\mathbbm{1};\leq 1}(\op M, \op M')=\TopCat(\op M(1),\op M'(1)).
\]

 Hence we find
\[
\Bimod_{W\op P,\leq 1}(B\op P_{\leq 1},(\op Q{\circ} X)_{\leq 1})
\simeq
\Bimod_{\mathbbm{1},\leq 1}(\mathbbm{1}_{\leq 1},(\op Q{\circ} X)_{\leq 1})
\simeq
(\op Q{\circ} X)(1) = \op Q(1) \times X.
\]
Here the map to the right-hand side is given by taking the image of the unit element.

By essentially the same argument we show that the map $\Bimod_{W\op P,X,\leq 1}(B\op P,\op Q) \to \op Q(1) \times X$ is a weak equivalence. This then shows the Lemma.
\end{proof}

\begin{rem}\label{r:framed_bim_oper}
We briefly explain~\eqref{eq:fr_nfr} and~\eqref{eq:fr_nfr2}. One has that $\op P$ and $\op Q$ are reduced, $\op P(1)\simeq \op Q(1)\simeq *$, $G$ is connected
and $G\nsimeq *$.   For simplicity, assume that $\op P(1)=*$. 
Recall~\eqref{eq:tr_op_map},~\eqref{eq:tr_bim_map} and the pullback square~\eqref{eq:pullback1}.  
 For the 1-truncated derived mapping spaces 
\[
\Bimod_{\op P;\leq 1}^h(\op P_{\leq 1},(\op Q{\circ} G)_{\leq 1})\simeq G\nsimeq \Bimod_{\op P;\leq 1}^h(\op P_{\leq 1},\op Q_{\leq 1})\simeq *,
\]
while
\[
\Operad_{\leq 1}(\op P_{\leq 1},\op (Q{\circ} G)_{\leq 1})\simeq*\simeq \Operad_{\leq 1}(\op P_{\leq 1},\op Q_{\leq 1}).
\]
On the other hand, for any $k\geq 2$, one has weak equivalences of fibers
\begin{multline*}
\mathrm{fiber}\bigl(\Bimod_{\op P;\leq k}^h(\op P_{\leq k},(\op Q{\circ} G)_{\leq k})\to \Bimod_{\op P;\leq k-1}^h(\op P_{\leq k-1},(\op Q{\circ} G)_{\leq k-1})
\bigr) \simeq\\
 \mathrm{fiber}\bigl(\Bimod_{\op P;\leq k}^h(\op P_{\leq k},\op Q_{\leq k})\to \Bimod_{\op P;\leq k-1}^h(\op P_{\leq k-1},\op Q_{\leq k-1})
\bigr),
\end{multline*}
\begin{multline*}
\mathrm{fiber}\bigl(\Operad_{\leq k}^h(\op P_{\leq k},(\op Q{\circ} G)_{\leq k})\to \Operad_{\leq k-1}^h(\op P_{\leq k-1},(\op Q{\circ} G)_{\leq k-1})
\bigr) \simeq\\
 \mathrm{fiber}\bigl(\Operad_{\leq k}^h(\op P_{\leq k},\op Q_{\leq k})\to \Operad_{\leq k-1}^h(\op P_{\leq k-1},\op Q_{\leq k-1})
\bigr).
\end{multline*}
For example, for the first equivalence above, we notice that the fibers are described as the spaces of $\Sigma_k$-equivariant
 lifts in the squares
$$
\xymatrix{
\partial B \op P(k)\ar[r] \ar[d] & Q{\circ} G(k) \ar[d] \\
B \op P(k) \ar[r]\ar@{..>}[ru] & \mathcal{M}(Q{\circ} G)(k)
}
\quad\quad\quad 
\xymatrix{
\partial B \op P(k)\ar[r] \ar[d] & Q(k) \ar[d] \\
B \op P(k) \ar[r]\ar[r]\ar@{..>}[ru] & \mathcal{M}(Q)(k),
}
$$
respectively.
These fibers are equivalent, because for $\op Q$ Reedy fibrant, the fibers of the maps $\op Q(k)\to \mathcal{M}(Q)(k)$ and $\op Q{\circ}G(k)\to \mathcal{M}(Q{\circ}G)(k)$ differ only for $k=1$, see the proof of Lemma~\ref{l:circle}.
\end{rem}

\begin{thm}\label{THmPart1}
If $\op P(1)$ is contractible, then the map 
$$
\Bimod_{W\op P,X}(B\op P, \op Q)\longrightarrow \Bimod_{W\op P}(B\op P, \op Q{\circ} X)
$$
is a weak equivalence.
\end{thm}

\begin{proof}
We compare the two fibrations (cf. Proposition~\ref{lem:fib})
\[
\xymatrix{
Y_{1}\ar[r] \ar[d]& \Bimod_{W\op P,X}(B\op P,\op Q) \ar[r]^{f}\ar[d] & \Bimod_{W\op P,X}(B\op P_{1},\op Q) \ar[d]^{h}  \\
Y_{2}\ar[r] & \Bimod_{W\op P}(B\op P,\op Q{\circ} X) \ar[r]_{g} & \Bimod_{W\op P}(B\op P_{1},\op Q{\circ} X), 
}
\]
where $Y_{1}$ is the fiber $f^{-1}(\alpha)$ over a point $\alpha\in \Bimod_{W\op P,X}(B\op P_{1},\op Q)$ and $Y_{2}=g^{-1}(h(\alpha))$. In particular, for any arity $i$, $h(\alpha)(i)\colon B\op P_{1}(i)\to Q(i)\times X^{\times i}$ is the constant map  on the factor $X^{\times i}$ with value $(x,\ldots,x)$ for some $x\in X$. Thanks to the identification $u^{\ast}{\circ}  h(\alpha)(n)= h(\alpha)(1){\circ}  u^{\ast}$, for any $u\in \Lambda_{+}([1],[n])$, any point in the fiber $Y_{2}=g^{-1}(h(\alpha))$ is also constant with value $(x,\ldots,x)$ on the factor $X^{\times -}$ and the two fibers $Y_{1}$ and $Y_{2}$ coincide.
Since $\op P(1)$ is contractible by assumption we may use Lemma \ref{lem:T1} to conclude that the right-hand vertical map is a weak equivalence. Hence so must be the middle vertical map.
\end{proof}

\subsection{The map from the homotopy fiber}\label{sec:map_from_fib}

Furthermore, one has a natural map
\beq{equ:hofibermapdef}
\psi':\hofiber( X\to \Operad(W\op P,\op Q) ) \to \Bimod_{W\op P,X}(B \op P,\op Q ).
\eeq
First, an element of the homotopy fiber on the left-hand side is a pair $(x,g)$ with $x\in X$ and 
 a path $g$ in $\Operad(W\op P,\op Q)$ connecting $\delta_{\ast}$ (at $t=0$) and $\delta_{x}$ (at $t=1$). Concretely, $g$ is a family of continuous maps 
$$
g_n\colon W\op P(n)\times [0\,,\,1]  \longrightarrow \op Q(n),\,\,  n\geq 1,
$$
satisfying the relations:
\begin{itemize}
\item[$\blacktriangleright$] $g_{n}(\iota(\ast_{1})\,;\, t)=\ast_{1}'$\hspace{98pt} $\forall t\in [0\,,\,1]$,\vspace{5pt}
\item[$\blacktriangleright$] $g_{n+m}(y_1{\circ}_{i}y_2,t)=g_{n+1}(y_1,t){\circ}_{i}g_{m}(y_2,t),$ 
\hspace{10pt} $\forall t\in [0\,,\,1]$,  $y_1\in W\op P(n+1)$, $y_2\in W\op P(m)$ and $i\in \{1,\ldots , n+1\}$,
\vspace{5pt}
\item[$\blacktriangleright$] $g_{n}(y,0)=\delta_{\ast}(y)=\eta{\circ}\mu (y)$, \hspace{59pt}  $\forall y\in W\op P(n)$,
\vspace{5pt}
\item[$\blacktriangleright$] $g_{n}(y,1)=\delta_{x}(y)$, \hspace{97pt}  $\forall y\in W\op P(n)$,
\vspace{5pt}
\item[$\blacktriangleright$] $g_{m}(u^{\ast}(y),t)=u^*(g_{n}(y,t))$, \hspace{59pt}  $\forall t\in [0\,,\,1]$, $y\in W\op P(n)$ and $u\in \Lambda([m],[n])$.
\vspace{7pt}
\end{itemize}

Let $(x\,;\, g)$ be an element in the homotopy fiber and let $[T\,;\,\{t_{v}\}\,;\,\{x_{v}\}]$ be a point in $B\op P$. It is a tree $T$ with each vertex $v$ labelled by a pair $(x_v,t_v)$. 
  The map $\psi'$ sends $(x\,;\,g)$ to the pair $(x\,;\,\psi'(g))$, where 
  $\psi'(g)([T\,;\,\{t_{v}\}\,;\,\{x_{v}\}])$ is defined as follows. One replaces each label $(x_v,t_v)$
  by $g_{|v|}(x_v,t_v)$ and then one composes the new labels using the structure of $T$ and the
  composition maps of the operad  $\op Q$.
 For instance, the image of the point $[T\,;\,\{t_{v}\}\,;\,\{x_{v}\}]\in B\op P(6)$ associated to the operadic composition in Figure \ref{B2} is the following one:\vspace{3pt}
$$
\begin{array}{cl}\vspace{5pt}
\psi'(g)([T\,;\,\{t_{v}\}\,;\,\{x_{v}\}]) & = g_{2}(x_{1},t_1)\big( g_{3}(a,1)\,;\,g_{3}(x_{2},t_2)\big),  \\ 
 & =g_{2}(x_{1},t_1)\big( \delta_x(a)\,;\,g_{3}(x_{2},t_2)\big).
\end{array} \vspace{3pt}
$$

We will derive our main result \eqref{equ:main} from the following  statements.

\begin{lemma}
The map $\pi:\Bimod_{W\op P, X}(B\op P, Q)\to X$ is a Serre fibration.
\end{lemma}

\begin{proof}
The statement is an immediate consequence of Proposition~\ref{lem:fib} for $\ell=0$, $k=\infty$. Indeed, $B\op P_0(i)=
\emptyset$, 
 $i\geq 1$. Consequently, $\Bimod_{W\op P, X}(B\op P_0, Q)= X$.
\end{proof}

Now, our main result is the following.
\begin{thm}\label{ThmPart2}
If $Q(1)$ is weakly contractible and $X$ is a path-connected pointed space, then the following is a homotopy fiber sequence
\[
\Bimod_{W\op P,X}(B\op P, \op Q) \to X \to \Operad(WP, Q),
\]
and the weak equivalence $\psi'\colon\hofiber(X \to \Operad(WP, Q))\to \Bimod_{W\op P,X}(B\op P, \op Q)$ is the map of section \ref{sec:map_from_fib}.
\end{thm}

\begin{proof}
We compare the two (horizontal) homotopy fiber sequences
\[
\begin{tikzcd}
\Omega( \Operad^h(\op P,\op Q))\ar{r}\ar{d} {\simeq}&  \hofiber( X\to \Operad^h(\op P,\op Q)) \ar{r}\ar{d}{\psi'}& X \ar{d}{=} \\
 \Bimod_{W\op P}(B\op P,\op Q) \ar{r} & \Bimod_{W\op P,X}(B\op P,\op Q) \ar{r}& X
\end{tikzcd}
\]
The left-hand vertical arrow is a weak equivalence by Theorem \ref{thm:duc}, and so is the right-hand vertical arrow.
We conclude that the middle vertical arrow must be a weak equivalence as well.  
\end{proof}

\subsection{A weak equivalence of Swiss-Cheese algebras}\label{Final2}

The one dimensional Swiss-Cheese operad $\mathcal{SC}_{1}$ is a two coloured operad with set of colours $S=\{o\,,\,c\}$ introduced by Voronov~\cite{Voronov}. It is a relative version of the one dimensional little discs operad $\lD_{1}$ defined as follows:
$$
\mathcal{SC}_{1}(n,m;k):=\left\{
\begin{array}{ll}\vspace{3pt}
\lD_{1}(n) & \text{if } m=0 \text{ and } k=c, \\ \vspace{3pt}
\left\{ \{c_{i}:[0\,,\,1]\rightarrow [0\,,\,1]\}_{1\leq i \leq n+1}\in \lD_{1}(n+1) \,\big| \, c_{n+1}(1)=1 \right\} & \text{if } m=1 \text{ and } k=o, \\ 
\emptyset & \text{otherwise},
\end{array} 
\right.
$$
An algebra over $\mathcal{SC}_{1}$ is given by a pair of topological spaces $(A\,,\,B)$ such that $A$ is a $\lD_{1}$-algebra and $B$ is a left module over $A$. A typical example of $\mathcal{SC}_{1}$-algebra is a pair of spaces of the form 
$$
\left(\, \Omega Y \,;\, \Omega(Y\,;\,X)=\hofiber(f:X\rightarrow Y)\,\right),
$$ 
where $f:X\rightarrow Y$ is a map of pointed spaces. In particular, we are interested in the case $Y=\Operad(W\op P, \op Q)$ based on the composite map $\eta{\circ}\mu :W\op P\rightarrow \op P \rightarrow \op Q$. So, the pair 
$$
\left( \Omega \Operad(W\op P, \op Q) \,;\, \hofiber(X\rightarrow \Operad(W\op P, \op Q))\right)
$$
is a $\mathcal{SC}_{1}$-algebra. Moreover, in \cite[Section 2.3]{Duc}, we build an explicit $\lD_{1}$-algebra structure on the space $\Bimod_{W\op P}(B\op P,\op Q)$ making the maps (\ref{Mapcase1}) into weak equivalences of $\lD_{1}$-algebras. In this section, we extend this construction in order to get an explicit $\mathcal{SC}_{1}$-algebra structure on the pair 
\begin{equation}\label{SCalg}
\left(  \Bimod_{W\op P}(B\op P, \op Q) \,;\,  \Bimod_{W\op P}(B\op P, \op Q{\circ} X)\right).
\end{equation}
For this purpose we build maps 
$$
\alpha_{n,o}:\mathcal{SC}_{1}(n,1;o)\times \Bimod_{W\op P}(B\op P, \op Q)^{\times n} \times \Bimod_{W\op P}(B\op P, \op Q{\circ} X) \longrightarrow \Bimod_{W\op P}(B\op P, \op Q{\circ} X) 
$$
compatible with operadic structure of $\mathcal{SC}_{1}$.

\begin{figure}[!h]
\begin{center}
\includegraphics[scale=0.3]{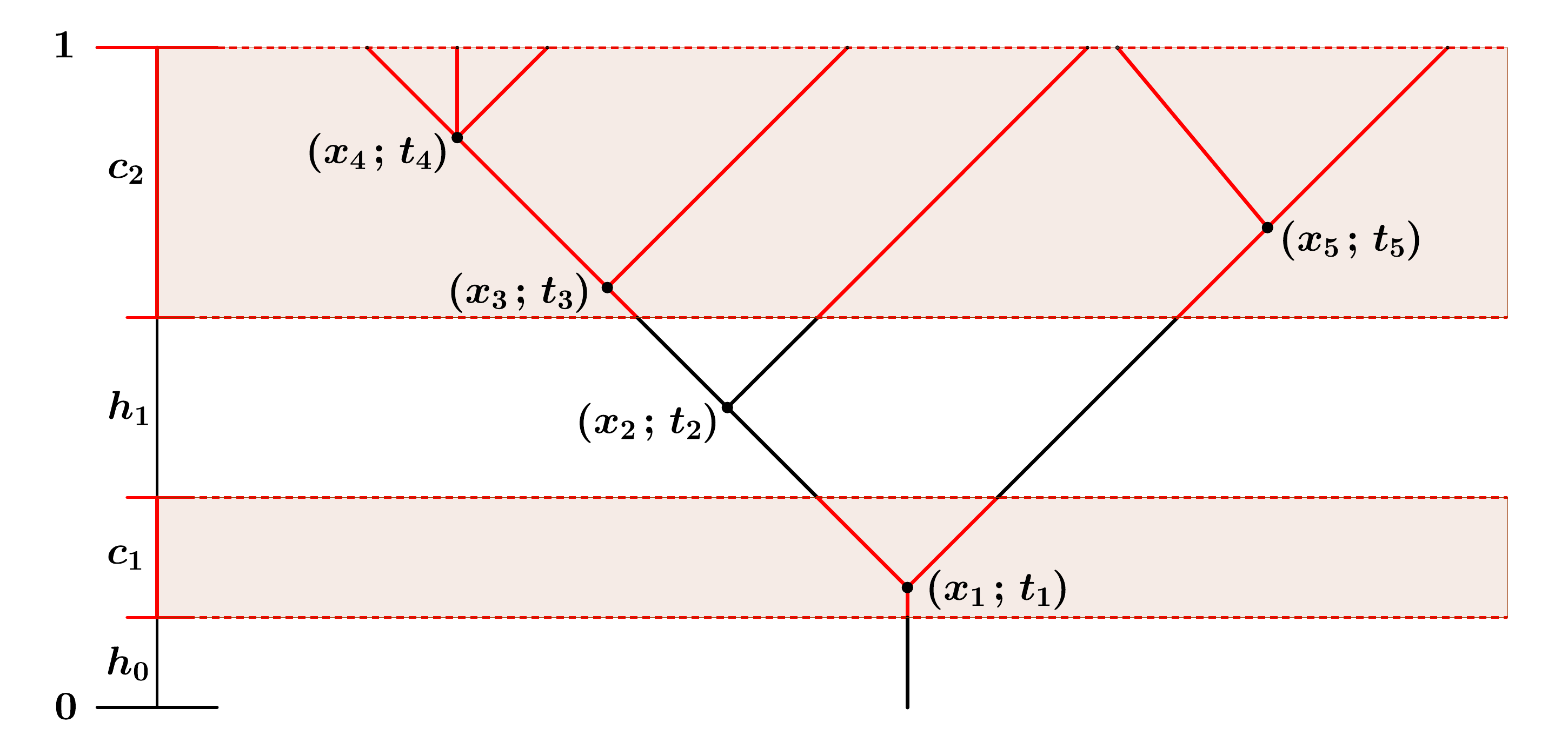}
\caption{Illustration of the subdivision of a point in $B\op P$ with the conditions\\  $c_{1}(0)< t_{1}<c_{1}(1)<t_{2}<c_{2}(0)<t_{3},t_{4},t_{5}<c_{2}(1)$.}\label{G0}
\end{center}
\end{figure}

From now on, we fix a family $c=\{c_{i}:[0\,,\,1]\rightarrow [0\,,\,1]\}_{1\leq i \leq n+1}\in \mathcal{SC}_{1}(n,1;o)$ as well as a family of bimodule maps $f_{i}:B\op P \rightarrow Q$, with $1\leq i \leq n$, and $f_{n+1}:B\op P\rightarrow Q{\circ} X$. Since the little discs arise from an affine embedding, $c_{i}$ is determined by the images of $0$ and $1$. In a similar way, we define the linear embeddings $h_{i}:[0\,,\,1]\rightarrow [0\,,\,1]$, with $0\leq i \leq n$, representing the gaps between the cubes:\vspace{3pt}
$$
h_{i}(0)=
\left\{
\begin{array}{cc}\vspace{2pt}
0 & \text{if } i=0, \\ 
c_{i}(1) & \text{if } i\neq 0,
\end{array} 
\right.
\hspace{15pt}\text{and}\hspace{15pt}
h_{i}(1)=
\left\{
\begin{array}{cc}\vspace{2pt}
1 & \text{if } i=n, \\ 
c_{i+1}(0) & \text{if } i\neq n.
\end{array} 
\right.\vspace{3pt}
$$

The bimodule map $\alpha_{n,o}(c\,;\,f_{1},\cdots,f_{n+1})$ is defined by using a decomposition of the points $y=[T\,;\, \{t_{v}\}\,;\,\{x_{v}\}]\in B\op P$ according to the parameters indexing the vertices. Roughly speaking, the  segments $<c_{1},\ldots,c_{n+1}>$ subdivide the tree $T$ into sub-trees as shown in Figure \ref{G0}. Then, we apply the bimodule map $f_{i}$ to the sub-trees associated to the  segment $c_{i}$ and the composite map $\eta{\circ}\mu:B\op P\rightarrow P\rightarrow Q$ to the sub-trees associated to gaps. Finally, we put together the pieces by using the operadic structure of $Q$ and the left $Q$-module
structure  of $Q{\circ} X$. By construction, we can assume that the representative point $y$ does not have two consecutive vertices (i.e. connected by an inner edge) indexed by the same real number. For the moment, we also assume that the
planar  tree $T$ is planar is labelled by the identity permutation.

More precisely, a sub-point of $y=[T\,;\,\{t_{v}\}\,;\, \{x_{v}\}]$ is an element in $B\op P$ obtained from $y$ by taking a sub-tree of $T$ preserving the indexation. A sub-point $w$ is said to be associated to the gap $h_{i}$ if the vertices below $w$ (seen as a sub-point of $y$) are strictly below the line $h_{i}(0)$ whereas the vertices above $w$ are strictly above $h_{i}(1)$. Furthermore, the parameters indexing the vertices of the main tree of $w$ are in the interval $[h_{i}(0)\,,\,h_{i}(1)]$. The set $\mathcal{T}[h_{i}\,;\,y]=\{w_{1}^{i},\ldots,w_{p_{i}}^{i}\}$ of sub-points associated to the gap $h_{i}$ is ordered using the planar structure of the tree $T$. For instance, the sets $\mathcal{T}[h_{0}\,;\,y]$ and $\mathcal{T}[h_{1}\,;\,y]$ associated to the point in Figure \ref{G0} are the following ones:
\begin{figure}[!h]
\begin{center}
\includegraphics[scale=0.2]{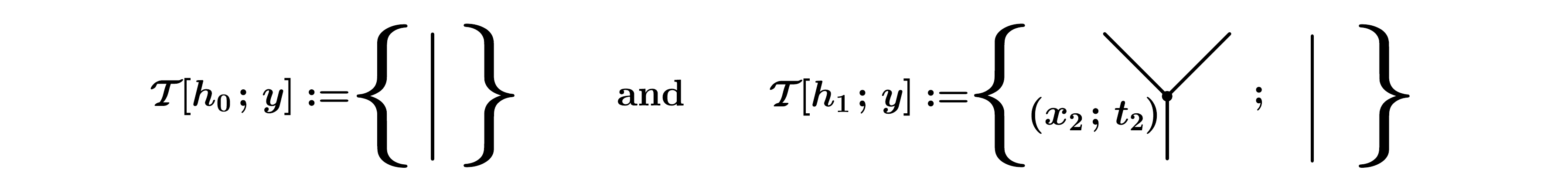}
\end{center}
\end{figure}

\noindent where the trivial tree without vertex represents the class of the $1$-corolla indexed by $(\iota(\ast_{1})\,;\,t)$ with $t\in [0\,,\,1]$ and $\ast_{1}$ the unit of the operad $\op P$.

Similarly, a sub-point $z$ is said to be associated to the segment $c_{i}$ if the vertices below $z$ (seen as a sub-point in $y$) are on the line $c_{i}(0)$ or below it, whereas the vertices above $z$ are on the line $c_{i}(1)$ or above it. Furthermore, the parameters indexing the vertices of the main tree of $z$ are in the interval $]c_{i}(0)\,,\,c_{i}(1)[$, if $i\leq n$, or in the interval $]c_{i}(0)\,,\,1]$, if $i=n+1$. The set $\mathcal{T}[c_{i}\,;\,y]=\{z_{1}^{i},\ldots,z_{q_{i}}^{i}\}$ of sub-points associated the little disc $c_{i}$ is ordered using the planar structure of the tree $T$. For instance, the sets $\mathcal{T}[c_{1}\,;\,y]$ and $\mathcal{T}[c_{2}\,;\,y]$ associated with the point $y$ in Figure \ref{G0} are the following ones:
\begin{figure}[!h]
\begin{center}
\includegraphics[scale=0.2]{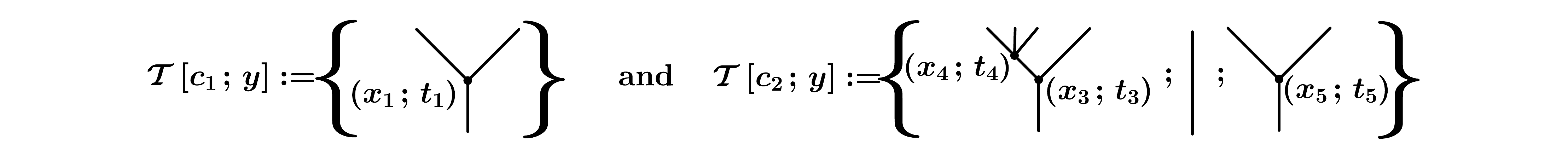}
\end{center}
\end{figure}

Let us remark that we  need the trivial trees in the above definition since the bimodule maps $\{f_{i}\}$ do not necessarily map the trivial tree to the unit of the operad $\op Q$. Furthermore, we need a map rescaling the parameters of the sub-points:
\begin{equation}\label{G9}
c_{i}^{\ast}:\mathcal{T}[c_{i}\,;\,y]\longrightarrow B\op P\,\,;\,\, [T'\,;\,\{t'_{v}\}\,;\, \{x'_{v}\}]\longmapsto [T'\,;\,\{c_{i}^{-1}(t'_{v})\}\,;\, \{x'_{v}\}].\vspace{-5pt}
\end{equation}
The map is well defined since the parameters indexing the vertices of the elements in $\mathcal{T}[c_{i}\,;\,y]$ are in the interval $]c_{i}(0)\,,\,c_{i}(1)[$ or $]c_{n+1}(0)\,,\,1]$. From the operadic structure of $\op Q$ and the left $W\op Q$-module structure on $\op Q{\circ} X$, we build the map $\alpha_{n,o}(c\,;\,f_{1},\cdots,f_{n+1})$ by induction as follows: \vspace{3pt}
$$
\begin{array}{lcl}\vspace{4pt}
\alpha_{n,o}(c\,;\,f_{1},\cdots,f_{n})_{0}(y) & = & \eta{\circ}\mu(w_{1}^{0}), \\ 
\alpha_{n,o}(c\,;\,f_{1},\cdots,f_{n})_{1}(y) & = & \alpha_{n,o}(c\,;\,f_{1},\cdots,f_{n})_{0}(y)\big(\,f_{1}(c_{1}^{\ast}(z_{1}^{1})),\ldots,f_{1}(c_{1}^{\ast}(z_{q_{1}}^{1}))\,\big),  \\ 
 & \vdots & \\ \vspace{4pt}
\alpha_{n,o}(c\,;\,f_{1},\cdots,f_{n})_{2k}(y) & = & \alpha_{n,o}(c\,;\,f_{1},\cdots,f_{n})_{2k-1}(y)\big(\,\eta{\circ}\mu(w_{1}^{k}) ,\ldots,\eta{\circ}\mu(w_{p_{k}}^{k}) \,\big), \\ 
\alpha_{n,o}(c\,;\,f_{1},\cdots,f_{n})_{2k+1}(y) & = & \alpha_{n,o}(c\,;\,f_{1},\cdots,f_{n})_{2k}(y)\big(\,f_{k}(c_{k}^{\ast}(z_{1}^{k})),\ldots,f_{k}(c_{k}^{\ast}(z_{q_{k}}^{k}))\,\big), \\ 
 & \vdots &  \\ 
\alpha_{n,o}(c\,;\,f_{1},\cdots,f_{n})(y) & = & \alpha_{n,o}(c\,;\,f_{1},\cdots,f_{n})_{2n}(y)\big(\,f_{n+1}(c_{n+1}^{\ast}(z_{1}^{n+1})),\ldots,f_{n+1}(c_{n+1}^{\ast}(z_{q_{n+1}}^{n+1}))\,\big).
\end{array} 
$$
We do not need to rescale  the sub-points associated to gaps since the map $\mu:B\op P\rightarrow \op P$ sends all the parameters indexing the vertices to $0$. 

This construction produces also a $\mathcal{C}_{1}$-algebra structure on $\Bimod_{W\op P}(B\op P, \op Q)$.
Note also that the $\mathcal{SC}_1$-action restricts on the pair 
$$
\left(  \Bimod_{W\op P}(B\op P_k, \op Q) \,;\,  \Bimod_{W\op P}(B\op P_k, \op Q{\circ} X)\right),
$$
 because the sub-points of an element in $B\op P_{k}$ are still elements in $B\op P_{k}$ and the rescaling maps (\ref{G9}) decrease the number of geometrical inputs. 
%

One has the following statement:

\begin{thm}
The morphisms induced by (\ref{Mapcase1}) and (\ref{Mapcase2})
$$
\left\{\begin{array}{c}\vspace{4pt}
\Omega \Operad(W\op P , \op Q) \\ 
hofiber(X\rightarrow \Operad(W\op P , \op Q))
\end{array}\right\}\longrightarrow
\left\{\begin{array}{c}\vspace{4pt}
 \Bimod_{W\op P}(B\op P , \op Q) \\ 
 \Bimod_{W\op P}(B\op P , \op Q{\circ} X)
\end{array}\right\}, 
$$
$$
\left\{\begin{array}{c}\vspace{4pt}
\Omega \Operad(W\op P_k , \op Q) \\ 
hofiber(X\rightarrow \Operad(W\op P_k , \op Q))
\end{array}\right\}\longrightarrow
\left\{\begin{array}{c}\vspace{4pt}
 \Bimod_{W\op P}(B\op P_k , \op Q) \\ 
 \Bimod_{W\op P}(B\op P_k , \op Q{\circ} X)
\end{array}\right\} 
$$
are morphisms of $\mathcal{SC}_{1}$-algebras. Furthermore, if $X$ is path-connected and the spaces $P(1)$ and $Q(1)$ are weakly contractible, then these are weak equivalences of $\mathcal{SC}_{1}$-algebras (for $k\geq 1$).
\end{thm}

\section{The smoothing theory delooping of $\Emb^\fr_\p(D^m,D^n)$}
\label{s:last}

The space $\Emb_\p(D^m,D^n)$, $n-m\geq 3$, $n\geq 5$, is known to have a delooping by means
of the smoothing theory~\cite[Proposition~1.3]{Sakai}:
\[
\Emb_\p(D^m,D^n)\simeq\Omega^m\hofiber(\St_{m,n}\to \St^t_{m,n}),
\]
where 
\[
\St_{m,n}^t=\TOP(n)/\TOP(n,m)
\] 
denotes the topological Stiefel manifold; $\TOP(n)$ is the group of homeomorphisms of $\R^n$;
$\TOP(n,m)$ is its subgroup of homeomorphisms preserving poinwise $\R^m\subset \R^n$.

\begin{prop}\label{p:deloop_sakai}
For $n-m\geq 3$, $n\geq 5$, one has
\beq{eq:deloop_sakai}
\Emb^\fr_\p(D^m,D^n)\simeq\Omega^{m+1}\bigl( \St^t_{m,n}\sslash\SO(n)\bigr).
\eeq
\end{prop}

\begin{proof}
One has a commutative diagram
\[
\begin{tikzcd}
\Emb^\fr_\p(D^m,D^n)\ar{r}\ar{d} &  \Omega^m \SO(n)\ar{r}\ar{d}& \Omega^m\St^t_{m,n}\ar{d}{=} \\
  \Emb_\p(D^m,D^n)\ar{r} & \Omega^m\St_{m,n} \ar{r}& \Omega^m\St^t_{m,n}.
\end{tikzcd}
\]
The lower line is a fiber sequence by \cite[Proposition~1.3]{Sakai}. The right vertical line
being identity, the middle one being a fibration, and the left square being a pullback one, all together imply that
the upper line is  also a fiber sequence. One gets
\[
\Emb^\fr_\p(D^m,D^n) \simeq\Omega^m\hofiber\bigl(\SO(n)\to \St^t_{m,n}\bigr)
\simeq \Omega^{m+1}\bigl( \St^t_{m,n}\sslash\SO(n)\bigr).
\]
\end{proof}

\begin{rem}\label{r:sakai1}
\sloppy
Note that $\St^t_{m,n}$ has a left action by $\SO(n)$. Thus by $\St^t_{m,n}\sslash\SO(n)$ we understand the
space $$\SO(n)\bbslash\TOP(n)/\TOP(n,m).$$
\end{rem}

\begin{rem}\label{r:sakai2}
Note that the same argument can be used to show that our delooping~\eqref{eq:cor2} easily follows from the 
Boavida-Weiss result~\eqref{eq:BW1}.
\end{rem}

\renewcommand{\thethm}{\arabic{thm}}

\end{document}